\documentclass[leqno,a4paper]{article}
\usepackage[english]{babel}
\usepackage{amsmath}
\usepackage{amsthm}
\usepackage{amssymb}
\usepackage{enumerate}
\usepackage{xspace}
\usepackage{euscript}
\usepackage{graphicx}
\usepackage{graphics}
\usepackage{amscd}
\usepackage{epsfig}
\usepackage{tabularx}
        \newtheorem{lemma}{Lemma}[section]
        \newtheorem{proposition}[lemma]{Proposition}
         \newtheorem{theorem}[lemma]{Theorem}
        \newtheorem{definition}{Definition}[section]
        \newtheorem{remark}[lemma]{Remark}

\numberwithin{equation}{section}
\title{\bf{Lipschitz stability for a piecewise linear Schr\"{o}dinger potential from local Cauchy data}}
\author{Giovanni Alessandrini\thanks{Dipartimento di Matematica e Geoscienze, Universit\`{a} di Trieste, Italy. Email:alessang@units.it}\qquad{Maarten V. de Hoop\thanks{Departments of Computational and Applied Mathematics, Earth Science, Rice University, Houston, Texas, USA. Email:mdehoop@rice.edu}}\qquad\\
Romina Gaburro\thanks{Department of Mathematics and Statistics, University of Limerick, Ireland.  Email: romina.gaburro@ul.ie}\qquad
Eva Sincich\thanks{Dipartimento di Matematica e Geoscienze, Universit\`{a} di Trieste, Italy. Email:esincich@units.it}}

\date{}

\begin{document}
\maketitle

\begin{abstract}
We consider the inverse boundary value problem of determining the potential $q$ in the equation $\Delta u + qu = 0$ in $\Omega\subset\mathbb{R}^n$, from local Cauchy data.  A result of global Lipschitz stability is obtained in dimension $n\geq 3$ for potentials that are piecewise linear on a given partition of $\Omega$. No sign, nor spectrum condition on $q$ is assumed, hence our treatment encompasses the reduced wave equation $\Delta u + k^2c^{-2}u=0$  at fixed frequency $k$.
\end{abstract}

\section{Introduction}\label{sec1}
The purpose of this paper is to achieve good stability estimates in the determination of the coefficient $q=\frac{1}{c^2}$ ($c$=wavespeed) in the Helmholtz type equation

\begin{equation}\label{Helmholtz eq}
\Delta u + k^2 qu = 0
\end{equation}

in a domain $\Omega$, when all possible Cauchy data $u\big\vert_{\Sigma}, \frac{\partial u}{\partial\nu}\big\vert_{\Sigma}$ are known on an open portion $\Sigma$ of $\partial\Omega$, at a single given frequency $k$. In view of the well-known exponential ill-posedness of this problem \cite{Ma} we shall introduce the rather strong \textit{a-priori} assumption on the unknown coefficient $q$ of being piecewise linear. The precise formulation will be given later on in section \ref{sec2}.
Uniqueness of the inverse boundary value problem associated with the Helmholtz equation in dimension $n\geq 3$ was established by Sylvester and Uhlmann \cite{Sy-U-1} assuming that the wavespeed is a bounded measurable function.

The inverse boundary value problem associated with the Helmholtz equation has
been extensively studied from an optimization point of view primarily
using computational experiments. In reflection seismology, iterative
methods for this inverse boundary value problem have been collectively
referred to as full waveform inversion (FWI). (The term `full waveform
inversion' was supposedly introduced by Pan, Phinney and Odom
in \cite{Pan1988} with reference to the use of full seismograms
information). Lailly \cite{Lailly1983} and
Tarantola \cite{Tarantola1984, Tarantola1987} introduced the
formulation of the seismic inverse problem as a local optimization
problem using a misfit functional. The misfit functional was
originally based on a least-squares criterion, but it needs to be
carefully designed to fit the analysis of the inverse boundary value
problem on the one hand - possibly, using Hilbert-Schmidt operators as
the data - and match actual data acquisition on the other
hand. Furthermore, we mention the original work of Bamberger,
Chavent \& Lailly \cite{Bamberger1977, Bamberger1979} in the
one-dimensional case. Initial computational experiments in the
two-dimensional case were carried out by Gauthier \cite{Gauthier1986}.

% Since then, a range of alternative misfit functionals have been
% considered; we mention, here, the criterion derived from the
% instantaneous phase used by Bozdag, Trampert and
% Tromp \cite{Bozdag2011}.

The time-harmonic formulation was initially promoted by Pratt and his
collaborators \cite{Pratt1990, Pratt1991}; he also emphasized the
importance of available wide-angle reflection data
in \cite{Pratt1996}. In this line of research, we mention the more
recent work of Ben-Hadj-Ali, Operto and
Virieux \cite{Ben-Hadj-Ali2008}.

A key question, namely, that of convergence of such iterative schemes
in concert with a well-understood regularization, was left open for
almost three decades. It appears that Lipschitz stability estimates
and conditional Lipschitz stability estimates provide a platform for
convergence analyses of the Landweber iteration \cite{dH-Q-S-1} and
projected steepest descent method \cite{dH-Q-S-2}, respectively, with
natural extensions to Newton-type methods.

As mentioned above, the exponential character of instability of the inverse boundary value
problem associated with the Helmholtz equation cannot be avoided. However, conditional Lipschitz stability
estimates can be obtained: Including discontinuities in the
coefficient, Beretta, De Hoop and Qiu \cite{Beretta2012} showed that
such an estimate holds if the unknown coefficient is a piecewise constant function with a known underlying
domain partition. Beretta, De Hoop, Qiu and Scherzer \cite{Beretta2016} also give a quantitative estimate for the
stability constant revealing the precise exponential growth with the
number of subdomains in the partition. We generalize the conditional
Lipschitz stability estimate to piecewise linear functions. Moreover, we use Cauchy data rather than
the Dirichlet-to-Neumann map. If we view to \eqref{Helmholtz eq} as the reduced wave equation, the coefficient $q(x)$ equals $\frac{1}{c^2(x)}$, where $c$ is the variable speed of propagation and thus $q>0$. This implies that $0$ might be a Dirichlet (or Neumann) eigenvalue, and even if not, it might be close to an eigenvalue. Therefore it is not convenient for the purpose of stability estimates, to express the boundary data in terms of the well-known D-N map (or the N-D one). We have chosen to express errors on the boundary data in terms of the so-called angle (or distance) between the spaces of Cauchy data viewed as subspaces of a suitable Hilbert space (see below section \ref{sec2}).

Moreover, since the present estimates are obtained for measurements at one fixed frequency $k>0$ we convene from now on to set $k=1$ and we shall also admit that $q$ may be real valued but of variable sign, thus our analysis encompasses more generally the stationary Schr\"odinger equation

\[\Delta u+qu=0.\]

Let us emphasize also that Cauchy data are a proxy to data obtained from advanced marine
acquisition systems. Here, airgun arrays excite waves underneath the
sea surface that is accounted for by a Dirichlet boundary condition,
which are detected in possibly variable depth towed dual sensor
streamers positioned (on some hypersurface) below the airgun
arrays. Dual sensors provide both the pressure and the normal particle
velocity forming Cauchy data. To allow favorable paths of (parallel)
streamers, in so-called full-azimuth acquisition, one uses two
recording vessels with their own sources and two separate source
vessels~\footnote{See, for example, the dual coil shooting
full-azimuth acquisition by WesternGeco.}

One can exploit conditional Lipschitz stability estimates, via a
Fourier transform, in the corresponding time-domain inverse boundary
value problem with bounded frequency data. Datchev and De Hoop
\cite{Datchev2015} showed how, via resolvent estimates for the
Helmholtz equation, the prerequisites for application of projected
steepest descent and Newton-type iterative reconstruction methods to
inverse wave problems can be satisfied.

%The nonlinearity in the inverse boundary value problem has motivated
%studies to develop hierarchical multiscale strategies which one may
%identify with iterative regularization. In the bounded frequency
%formulation, Bunks \textit{et al.} \cite{Bunks1995} proposed
%successive inversion of data subsets of increasing frequency contents
%to mitigate so-called cycle skipping. This multiscale approach can,
%for example, be related to the subspace search method advocated by
%Kennett, Sambridge \& Williamson \cite{Kennett1988}. The idea of
%frequency progression, in a time-harmonic formulation, has been
%considered, amongst others, by Sirgue \& Pratt \cite{Sirgue2004} and
%extensively studied in the electromagnetic waves case, for example, by
%Bao and Li \cite{Bao2005, Bao2009}.

% Malinowski, Operto \& Ribodetti \cite{Malinowski2011} demonstrate
% high-resolution imaging of attenuation and phase velocity in the
% visco-acoustic case. In this context, we also mention the work of
% Askan, Akcelik, Bielak \& Ghattas \cite{Askan2007}. We note that our
% convergence analysis applies to complex wavespeeds.

With the objective of obtaining approximate reconstructions in the
class of models for which conditional Lipschitz stability estimates
hold, from a starting model -- the error of which can be estimated as
well in this context \cite{dH-Q-S-1} -- compression plays an
important role. This is elucidated in the work of De Hoop, Qiu and
Scherzer \cite{dH-Q-S-2} using multi-level schemes based on successive
refinement and arises in mitigating the growth of the stability
constants with the number of subdomains in the partition or, simply,
the number parameters. Indeed, the class of piecewise linear functions
provide an excellent way to achieve compression in the presence of
discontinuities. The application of wavelet bases in compressing the
successive models in iterative methods has been considered
by \cite{Loris2007, Loris2010} in wave-equation tomography in a the
framework of sparsity promoting optimization and in FWI by Lin,
Abubakar \& Habashy \cite{Lin2012} for the purpose of reducing the
size of the Jacobian.

In FWI one commonly applies a `nonlinear' conjugate gradient method, a
Gauss-Newton method, or a quasi-Newton method (L-BFGS; for a review,
see Brossier \cite{Brossier2009}). For the application of multi-scale
Newton methods, see Akcelik \cite{Akcelik2002}. In the gradient
method, the step length is typically estimated by a simple line search
for which a linearization of the direct problem is used (Gauthier,
Virieux \& Tarantola \cite{Gauthier1986}). This estimation is
challenging in practice and may lead to a failure of convergence. For
this purpose, research based on trust region has been studied by
Eisenstat \& Walker in \cite{Eisenstat1994}, for FWI see M\'etivier
and others in \cite{M'etivier}, the method is detailed
in \cite{Conn2000}. In various approaches based on the Gauss-Newton
scheme one accounts just for the diagonal of the
Hessian \cite{Shin2001}. In certain earthquake seismology
applications, one builds the Fr\'{e}chet derivative or Jacobian
(sensitivity, for example by Chen \textit{et al.} \cite{Chen2007})
explicitly and then applies LSQR. We also mention the work of
M\'etivier and others in \cite{M'etivier} based on Hessian
vector multiplication techniques to reduce the cost of a dense Hessian
computation. The use of complex frequencies was studied
in \cite{Shin2009, Ha2010}.

Coming back to the object of the present paper, let us point out that various new aspects appear in comparison to prior results of stability under assumptions of piecewise constant or piecewise linear coefficients \cite{A-V, A-dH-G-S, Beretta2012} which require novel arguments.

\begin{enumerate}[I)]

\item\underline{Singular solutions.} Since we admit that the underlying equation may be in the eigenvalue regime for the Dirichlet or the Neumann boundary value problem we need to construct Green's functions for a boundary of mixed type which is of Dirichlet type on part of the boundary and of complex valued Robin-type on the remaining part. This construction relies on quantitative estimates of unique continuation which take inspiration from an idea of Bamberger and Hua Duong \cite{Ba-Du} and an iterative procedure for the approximation with the standard fundamental solution of Laplace's equation.

\item\underline{Asymptotics.} In order to determine values of the potential $q$ and its gradient we need singular solutions whose blow up rate is of the order up to $|x|^{-n}$, for this purpose we have to estimate the asymptotic behaviour of the previously found Green's function up to its second derivatives.

\item\underline{Stability at the boundary.} The typical initial step in stability from inverse boundary value problems is the stability at the boundary of the unknown coefficient. For Calder\'on's problem it is well-known that the stability at the boundary for the conductivity coefficient is of Lipschitz type \cite{Sy-U-2, A1}. For the potential coefficient $q$ the situation is different. In fact we are able in general to obtain only H\"older type stability. This fact is related to the different dimensionality of the boundary energy $\int_{\partial\Omega} u\:\frac{\partial u}{\partial\nu}$ and of the volume integral $\int_{\Omega} qu^2$ whereas in the conductivity case ($\mbox{div}(\gamma\nabla u) =0$) the equality $\int_{\partial\Omega} u\gamma\:\frac{\partial u}{\partial\nu}=\int_{\Omega} \gamma |\nabla u|^2$ provides the right balance.

Nevertheless, under the piecewise linear assumption, using the kind of bootstrap argument introduced in \cite{A-V} we eventually achieve the desired global Lipschitz stability.

\end{enumerate}

The outline of the paper is as follows. In section \ref{sec2} we provide the basic set up. We introduce the spaces of local Cauchy data and their metric structure as subspaces of a Hilbert space in subsection \ref{subsec notation and definitions}. Next, in subsection \ref{subsection assumptions} we present the \textit{a-priori} assumptions on the domain and on the potential and we state our main stability result theorem \ref{teorema principale}. In subsection \ref{PLP} we state the main propositions which provide the main tools for the proof of theorem \ref{teorema principale} which is completed in subsection \ref{LSPLP}. Section \ref{PP} contains the proofs of the various propositions previously stated. Subsection \ref{AE} contains the construction of the Green's function for the mixed Dirichlet-Robin boundary value problem and the estimates of its asymptotic behavior. Subsection \ref{PS} contains the quantitative estimates of unique continuation adapted for the singular solution $\tilde{S}_{\mathcal{U}_k}$ introduced in Section \ref{PMR}. We conclude in subsection \ref{stability at the boundary} with the stability estimates at the boundary for the potential $q$ and its normal derivative.

%%%%%%%%%%%%%%%%%%%%%%%%%%          SECTION 2       %%%%%%%%%%%%%%%%%%%%%%%%%%%%

\section{Main Result}\label{sec2}
\setcounter{equation}{0}
\subsection{Definitions and preliminaries}\label{subsec notation and definitions}

In several places within this manuscript it will be useful to single out one coordinate
direction. To this purpose, the following notations for
points $x\in \mathbb{R}^n$ will be adopted. For $n\geq 3$,
a point $x\in \mathbb{R}^n$ will be denoted by
$x=(x',x_n)$, where $x'\in\mathbb{R}^{n-1}$ and $x_n\in\mathbb{R}$.
Moreover, given a point $x\in \mathbb{R}^n$,
we shall denote with $B_r(x), B_r'(x)$ the open balls in
$\mathbb{R}^{n},\mathbb{R}^{n-1}$ respectively centred at $x$ with radius $r$
and by $Q_r(x)$ the cylinder

\[Q_r(x)=B_r'(x')\times(x_n-r,x_n+r).\]

%We shall also denote

%\begin{eqnarray*}
%& & \mathbb{R}^n_+ = \{(x',x_n)\in \mathbb{R}^n| x_n>0 \};\quad\mathbb{R}^n_- = \{(x',x_n)\in \mathbb{R}^n| x_n<0 \};\\
%& & B^+_r = B_r\cap\mathbb{R}^n_+;\quad B^-_r = B_r\cap\mathbb{R}^n_-;\\
%& & Q^+_r = Q_r\cap\mathbb{R}^n_+;\quad Q^{-}_r = Q_r\cap\mathbb{R}^n_-,
%\end{eqnarray*}

%where we understand $B_r=B_r(0)$ and $Q_r=Q_r(0)$.
In the sequel, we shall make a repeated use of quantitative
notions of smoothness for the boundaries of various domains. Let
us introduce the following notation and definitions.

\begin{definition}\label{def Lipschitz boundary}
Let $\Omega$ be a domain in $\mathbb R^n$. We say that a portion
$\Sigma$ of $\partial\Omega$ is of Lipschitz class with constants
$r_0,L$ if there exists $P\in\Sigma$ and there exists a rigid
transformation of $\mathbb R^n$ under which we have $P=0$ and
$$\Omega\cap Q_{r_0}=\{x\in Q_{r_0}\,|\,x_n>\varphi(x')\},$$
$$\Sigma=\{x\in Q_{r_0}\,|\,x_n=\varphi(x')\},$$
where $\varphi$ is a Lipschitz function on $B'_{r_0}$ satisfying

\[\varphi(0)=|\nabla_{x'}\varphi(0)|=0;\qquad
\|\varphi\|_{C^{0,1}(B'_{r_0})}\leq Lr_0.\]

It is understood that $\partial\Omega$ is of Lipschitz class with
constants $r_0,L$ if it is finite union of portions of Lipschitz class with constants $r_0,L$.
\end{definition}

\begin{definition}\label{flat portion}
Let $\Omega$ be a domain in $\mathbb R^n$. We say that a portion $\Sigma$ of
$\partial\Omega$ is a flat portion of size $r_0$
if there exists $P\in\Sigma$ and there exists a rigid transformation of
$\mathbb R^n$ under which we have $P=0$ and

% prenderei i cilindretti di raggio e altezza r_0 e non r_0/3

%\begin{equation}
%\Sigma\cap{Q}_{r_{0}} =\{x\in
%Q_{r_0}|x_n=0\}.
%\end{equation}

\begin{eqnarray}\label{flat}
\Sigma\cap{Q}_{r_{0}/3} &=&\{x\in
Q_{r_0/3}\,|\,x_n=0\}\nonumber\\
\Omega\cap {Q}_{r_{0}/3} &=&\{x\in
Q_{r_0/3}\,|\,x_n>0\}\nonumber\\
\left(\mathbb{R}^{n}\setminus\Omega\right)\cap {Q}_{r_{0}/3} &=&\{x\in
Q_{r_0/3}\,|\,x_n<0\},
\end{eqnarray}
\end{definition}

%%%%%%%%%%%%%%%%%%%%%%%%%%%%%%%%% definizione D-N map locale %%%%%%%%%%%%%%%%%%%%%%%%%%%%%%%

Let us define the space of \emph{local Cauchy data} on $\Sigma$ for $H^{1}(\Omega)$ solutions to

\begin{equation}\label{Schrodinger eq 2}
\Delta u +qu=0\qquad\textnormal{in}\quad\Omega ,
\end{equation}

having zero trace on $\partial\Omega\setminus\overline\Sigma$.

%We start by defining the Cauchy data whose first component (the Dirichlet data) is zero outside $\Sigma$.

\begin{definition}\label{H1/2 00 and dual}
Let $\Omega$ be a domain in $\mathbb{R}^n$ with Lipschitz boundary $\partial\Omega$ and $\Sigma$ a non-empty open portion of $\partial\Omega$. Let us introduce the subspace of $H^{\frac{1}{2}}(\partial\Omega)$

\begin{equation}\label{Hco}
H^{\frac{1}{2}}_{co}(\Sigma)=\big\{f\in
H^{\frac{1}{2}}(\partial\Omega) \:\vert\:\textnormal{supp}
\:f\subset\Sigma\big\}.
\end{equation}

We recall that its closure with respect to the $H^{\frac{1}{2}}(\partial\Omega)$ - norm is the space $H^{\frac{1}{2}}_{00}(\Sigma)$ (see \cite{Lio-M}, \cite{Tartar}).

%We define the \textbf{Cauchy data} associated to $q$  to be the space $\mathcal{C}_q (\partial\Omega)$ defined by

%\begin{eqnarray}\label{full Cauchy data with zero trace condition}
%\mathcal{C}_q (\partial\Omega)=\Big\{(f,g)\in H^{\frac{1}{2}}(\partial\Omega)\times H^{-\frac{1}{2}}(\partial\Omega) & \big| &\:\exists u\in H^{1}(\Omega)\:\textnormal{weak\:solution\:to}\nonumber\\
%& & \Delta u+qu=0\:\textnormal{in}\:\Omega,\nonumber\\
%& & u\Big|_{\partial\Omega}=f,\: \frac{\partial u}{\partial\nu}\Big|_{\partial\Omega}=g\Big\}.
%\end{eqnarray}

We define the \emph{Cauchy data} associated to $q$ with first component vanishing on $\partial\Omega\setminus\overline\Sigma$ to be the space $\mathcal{C}^{\Sigma}_q (\partial\Omega)$ defined by

\begin{eqnarray}\label{full Cauchy data with zero trace condition}
\mathcal{C}^{\Sigma}_q (\partial\Omega)=\Big\{(f,g)\in H^{\frac{1}{2}}_{00}(\Sigma)\times H^{-\frac{1}{2}}(\partial\Omega) & \big| &\:\exists u\in H^{1}(\Omega)\:\textnormal{weak\:solution\:to}\nonumber\\
& & \Delta u+qu=0\:\textnormal{in}\:\Omega,\nonumber\\
& & u\Big|_{\partial\Omega}=f,\: \partial_{\nu} u \Big|_{\partial\Omega}=g\Big\}.
\end{eqnarray}

\end{definition}

Analogously, we consider the subspace of $H^{\frac{1}{2}}(\partial\Omega)$,

\[H^{\frac{1}{2}}_{00}(\partial\Omega\setminus\overline\Sigma)\]

and the closed subspace of $H^{-\frac{1}{2}}(\partial\Omega)$ of functionals vanishing on $H^{\frac{1}{2}}_{00}(\Sigma)$ functions

\begin{equation}\label{Hco outside Sigma}
H^{-\frac{1}{2}}_{00}(\partial\Omega\setminus\overline\Sigma)=\left\{\psi\in H^{-\frac{1}{2}}(\partial\Omega)\:|\:\langle \psi, \varphi\rangle =0,\quad\textnormal{for\:any}\:\varphi\in H^{\frac{1}{2}}_{00}(\Sigma)\right\}.
\end{equation}

Here $\langle\psi,\varphi\rangle$ denotes the duality between the complex valued spaces $H^{-\frac{1}{2}}(\partial\Omega)$, $H^{\frac{1}{2}}(\partial\Omega)$  based on the $L^2$ inner product

\[\langle\psi,\varphi\rangle=\int_{\partial\Omega} \psi\overline\varphi.\]

%Let

%\begin{eqnarray}\label{P1/2}
%& & P_{\frac{1}{2}}: H^{\frac{1}{2}}(\partial\Omega)\longrightarrow H^{\frac{1}{2}}_{00}(\partial\Omega\setminus\overline\Sigma)\\\label{P1/2}
%& & P_{-\frac{1}{2}}: H^{-\frac{1}{2}}(\partial\Omega)\longrightarrow H^{-\frac{1}{2}}_{00}(\partial\Omega\setminus\overline\Sigma)\label{P-1/2}
%\end{eqnarray}

%be the orthogonal projections of $H^{\frac{1}{2}}(\partial\Omega)$ and $H^{-\frac{1}{2}}(\partial\Omega)$ into their closed subspaces $H^{\frac{1}{2}}_{00}(\partial\Omega\setminus\overline\Sigma)$ and $H^{-\frac{1}{2}}_{00}(\partial\Omega\setminus\overline\Sigma)$ respectively.

We denote by $H^{\frac{1}{2}}(\partial\Omega)\big\vert_{\Sigma}$ and $H^{-\frac{1}{2}}(\partial\Omega)\big\vert_{\Sigma}$ the \emph{restrictions} of $H^{\frac{1}{2}}(\partial\Omega)$ and $H^{-\frac{1}{2}}(\partial\Omega)$ to $\Sigma$ respectively. Note that $H^{\frac{1}{2}}(\partial\Omega)\big\vert_{\Sigma}$ can be interpreted as the quotient space of $H^{\frac{1}{2}}(\partial\Omega)$ through the equivalence relation

\[\varphi\sim\psi\qquad\textnormal{iff}\qquad \varphi -\psi\in H^{\frac{1}{2}}_{00}(\partial\Omega\setminus\overline\Sigma).\]

The same reasoning applies to $H^{-\frac{1}{2}}(\partial\Omega)\big\vert_{\Sigma}$, that is

\[H^{-\frac{1}{2}}(\partial\Omega)\big\vert_{\Sigma}=H^{-\frac{1}{2}}(\partial\Omega)\Big/H^{-\frac{1}{2}}_{00}(\partial\Omega\setminus\overline\Sigma).\]
%\begin{eqnarray}
%H^{\frac{1}{2}}(\Sigma) &=& (I-P_{\frac{1}{2}}) H^{\frac{1}{2}}(\partial\Omega)\\\label{restriction}
%H^{-\frac{1}{2}}(\Sigma) &=& (I-P_{-\frac{1}{2}}) H^{-\frac{1}{2}}(\partial\Omega)\label{restriction dual}.
%\end{eqnarray}

%\begin{equation}\label{full Cauchy data}
%\mathcal{C}_q(\partial\Omega)=\Big\{\left( \begin{array}{c}
%u \\
%\frac{\partial u}{\partial\nu}\end{array} \right)\Big|_{\partial\Omega} \big| \: u\in H^{1}(\Omega)\:\textnormal{is\:a\:weak\:solution\:to}\:\Delta u+qu=0\quad\textnormal{in}\:\Omega\},
%\end{equation}

We can now define the local Cauchy data that will be considered here.

\begin{definition}\label{restrictions Cauchy data with zero trace condition}
The \emph{local Cauchy data} associated to $q$ \emph{having zero first component} on $\partial\Omega\setminus\overline\Sigma$ are defined by

%\begin{equation}\label{Cauchy data restricted}
%\mathcal{C}^{\Sigma}_q(\Sigma) = \left( \begin{array}{cc}
%I-P_{\frac{1}{2}} & 0\\
%0 & I-P_{-\frac{1}{2}}\end{array} \right)\mathcal{C}^{\Sigma}_q(\partial\Omega).
%\end{equation}

\begin{eqnarray}\label{local Cauchy data with zero trace condition 2}
\mathcal{C}^{\Sigma}_q (\Sigma)=\Big\{(f,g)\in H^{\frac{1}{2}}_{00}(\Sigma)\times H^{-\frac{1}{2}}(\partial\Omega)\big\vert_{\Sigma} & \big| &\exists u\in H^{1}(\Omega)\:\textnormal{weak\:solution\:to}\nonumber\\
& & \Delta u+qu=0\:\textnormal{in}\:\Omega,\nonumber\\
& & u\Big|_{\partial\Omega}=f,\nonumber\\
& &\langle\partial_{\nu} u \Big|_{\partial\Omega}, \varphi\rangle=\langle g,\varphi\rangle,\quad\forall \varphi\in H^{\frac{1}{2}}_{00}(\Sigma)\nonumber\Big\}.
\end{eqnarray}

%\begin{equation}\label{Cauchy data restricted}
%\mathcal{C}_q(\Sigma) = \left( \begin{array}{cc}
%I-P_{\frac{1}{2}} & 0\\
%0 & I-P_{-\frac{1}{2}}\end{array} \right)\mathcal{C}_q(\partial\Omega)

%\begin{remark}
%It can be shown that both  $C^{\Sigma}_q(\partial\Omega)$ and $C^{\Sigma}_q(\Sigma)$ are closed subspaces of $H^{\frac{1}{2}}(\partial\Omega)\times H^{-\frac{1}{2}}(\partial\Omega)$ and $H^{\frac{1}{2}}(\Sigma)\times %H^{-\frac{1}{2}}(\Sigma)$ respectively.
%\end{remark}

%\begin{remark}
%Notice that

%\begin{eqnarray}\label{local Cauchy data with zero trace condition 2}
%\mathcal{C}^{\Sigma}_q (\Sigma)=\Big\{(f,g)\in H^{\frac{1}{2}}_{00}(\Sigma)\times H^{-\frac{1}{2}}(\Sigma) & \big| &\exists u\in H^{1}(\Omega)\:\textnormal{weak\:solution\:to}\:\Delta u+qu=0\:\textnormal{in}\:\Omega,\nonumber\\
%& & u\Big|_{\partial\Omega}-f\in H^{\frac{1}{2}}_{00}(\partial\Omega\setminus\overline\Sigma),\nonumber\\
%& &\langle\frac{\partial u}{\partial\nu}\Big|_{\partial\Omega}, f\rangle=\langle g,f\rangle,\quad\forall f\in H^{\frac{1}{2}}_{00}(\Sigma)\nonumber\Big\}.
%\end{eqnarray}

%Here the restriction $\cdot\big|_{\Sigma}$ is to be meant

%\begin{equation}\label{restriction functionals 2}
%\left(\frac{\partial u}{\partial\nu} - g\right)\in H^{-\frac{1}{2}}_{00}(\partial\Omega\setminus\overline\Sigma).
%\end{equation}
%\end{remark}

For the sake of completeness, let us also introduce the general local Cauchy data $\mathcal{C}_q(\Sigma)$ with no zero Dirichlet condition on $\partial\Omega\setminus\overline\Sigma$

%The \textbf{local Cauchy data} associated to $q$ are defined as the \textbf{restrictions} to the portion $\Sigma$ of $\mathcal{C}_q(\partial\Omega)$

\begin{eqnarray}\label{local Cauchy data with zero trace condition 3}
\mathcal{C}_q (\Sigma)=\Big\{(f,g)\in H^{\frac{1}{2}}(\partial\Omega)\big\vert_{\Sigma}\times H^{-\frac{1}{2}}(\partial\Omega)\big\vert_{\Sigma}  & \big|  &\exists u\in H^{1}(\Omega)\:\textnormal{weak\:solution\:to}\nonumber\\
& & \Delta u+qu=0\:\textnormal{in}\:\Omega,\nonumber\\
& & u\Big|_{\partial\Omega}-f\in H^{\frac{1}{2}}_{00}(\partial\Omega\setminus\overline\Sigma),\nonumber\\
& &\langle\partial_{\nu} u \Big|_{\partial\Omega}, \varphi\rangle=\langle g,\varphi\rangle,\quad\forall \varphi\in H^{\frac{1}{2}}_{00}(\Sigma)\nonumber\Big\}.
\end{eqnarray}
\end{definition}

%In what follows we simply denote $\mathcal{C}^{\Sigma}_q (\Sigma)$ by $\mathcal{C}_q$.
Observe also that $H^{\frac{1}{2}}_{00}(\Sigma)\times H^{-\frac{1}{2}}(\partial\Omega)\big\vert_{\Sigma} $ is a Hilbert space with the norm

\begin{equation}\label{norm Cauchy data}
||(f,g)||_{H^{\frac{1}{2}}_{00}(\Sigma)\oplus  H^{-\frac{1}{2}}(\partial\Omega)\big\vert_{\Sigma}}=\left(||f||^{2}_{H^{\frac{1}{2}}_{00}(\Sigma)}+||g||^{2}_{H^{-\frac{1}{2}}(\partial\Omega)\big\vert_{\Sigma}}\right)^{\frac{1}{2}}.
\end{equation}

We recall that given closed subspaces $S_1, S_2$ of a Hilbert space $(H, ||\cdot||)$, the \emph{distance} (\emph{aperture}) between $S_1,S_2$ is defined as

\begin{equation}\label{Hausdorff distance}
d(S_1,S_2) =\max\left\{\sup_{h\in S_2, h\neq 0}\inf_{k\in S_1}\frac{||h-k||_{H}}{||h||_{H}}, \sup_{k\in S_1, k\neq 0}\inf_{h\in S_2}\frac{||h-k||_{H}}{||k||_{H}}\right\},
\end{equation}

see for instance \cite{Ak-Gl}. From now on, given two potential $q_i$, $i=1,2$, we will simply denote the local Cauchy data $\mathcal{C}^{\Sigma}_{q_i} (\Sigma)$ with $\mathcal{C}_i$, $i=1,2$. We recall also that it is known that when $d(S_1,S_2)<1$, then the two quantities within the maximum in \eqref{Hausdorff distance} coincide (see \cite{Kn-J-Ar}). Then, since we are interested in Cauchy data spaces $\mathcal{C}_{1}, \mathcal{C}_{2}$ corresponding to potentials $q_1, q_2$ respectively, when $\mathcal{C}_{1}$ and $\mathcal{C}_{2}$ are close to each other, it is sensible to set

\begin{equation}
d(\mathcal{C}_{1},\mathcal{C}_{2}) =\sup_{(f_2,g_2)\in\mathcal{C}_{2}}\inf_{(f_1,g_1)\in\mathcal{C}_{1}}\frac{||(f_1,g_1)-(f_2,g_2)||_{H^{\frac{1}{2}}_{00}(\Sigma)\oplus  H^{-\frac{1}{2}}(\partial\Omega)\big\vert_{\Sigma}}}{||(f_2,g_2)||_{H^{\frac{1}{2}}_{00}(\Sigma)\oplus  H^{-\frac{1}{2}}(\partial\Omega)\big\vert_{\Sigma}}}.
\end{equation}

Note also that properly speaking the above quantity is a distance between the closures $\overline{\mathcal{C}_{1}}$ and $\overline{\mathcal{C}_{2}}$. We notice that it could be proved that the subspaces $\mathcal{C}_{1}$ and $\mathcal{C}_{2}$ are indeed closed, but this fact is not much relevant in the present context.\\

Let us also recall some more or less well-known calculations. Let $u_i \in H^{1}(\Omega)$ be solutions to \eqref{Schrodinger eq 2} when $q=q_i$, $i=1,2$ respectively and such that $u_i\Big|_{\partial\Omega}\in H^{\frac{1}{2}}_{co}(\Sigma)$. Green's identity yields

\begin{equation}\label{Green's identity}
\int_{\Omega} (q_1-q_2) u_1 u_2 =\langle\partial_{\nu} \overline{u_2}, u_1\rangle - \langle \partial_{\nu} u_1, \overline{u_2}\rangle.
\end{equation}

Notice that a complex valued function $u_i$ is a solution to \eqref{Schrodinger eq 2} when $q=q_i$ (real valued) if and only if so is $\overline{u_i}$. Notice also that, if $v_i$ is any other such solution with $q=q_i$, we have

\begin{equation}
\langle \partial_{\nu} v_i,\overline{ u_i}\rangle - \langle \partial_{\nu} \overline{u_i}, v_i\rangle =0,\qquad\textnormal{for\:every}\quad i=1,2.
\end{equation}

Hence, for any $v_2$ solving $\Delta v_2+q_2v_2=0$ in $\Omega$,

\begin{equation}\label{Green's identity 2}
\int_{\Omega} (q_1-q_2) u_1 u_2 =\langle \partial_{\nu} \overline{u_2}, (u_1-v_2)\rangle - \langle \partial_{\nu} u_1 - \partial_{\nu} v_2, \overline{u_2}\rangle,
\end{equation}

from such an identity one easily deduces

\begin{eqnarray}\label{Alessandrini stability}
\left|\int_{\Omega} (q_1 -q_2)u_1u_2\:dx\right| &\leq & d(\mathcal{C}_{1},\mathcal{C}_{2})\left|\left|\left(u_1,\partial_{\nu} u_1\right)\right|\right|_{H^{\frac{1}{2}}_{00}(\Sigma)\oplus  H^{-\frac{1}{2}}(\partial\Omega)\big\vert_{\Sigma}}\\
& \times &\left|\left|\left({u_2},\partial_{\nu} \overline{u_2}\right)\right|\right|_{H^{\frac{1}{2}}_{00}(\Sigma)\oplus  H^{-\frac{1}{2}}(\partial\Omega)\big\vert_{\Sigma}}.\nonumber\\
\end{eqnarray}

\subsection{Conditional Lipschitz stability}\label{subsection assumptions}
%The assumptions regarding the domain $\Omega$ are the ones assumed
%in \cite{A-V}, where the \textit{a-priori} information on the
%conductivity $\sigma$ is slightly different since $\sigma$ is
%anisotropic here.\\

\subsubsection{Assumptions about the domain $\Omega$}\label{subsec assumption domain}

\begin{enumerate}

\item We assume that $\Omega$ is a domain in $\mathbb{R}^n$
and that there is a positive constant $B$ such that

\begin{equation}\label{assumption Omega}
|\Omega|\leq B r_0 ^n,
\end{equation}

where $|\Omega|$ denotes the Lebesgue measure of $\Omega$.

%\item We assume that $\partial\Omega$ is of Lipschitz class with
%constants $r_0$, $L$.

\item We fix an open non-empty subset $\Sigma$ of $\partial\Omega$
(where the measurements in terms of the local Cauchy data are taken).

\item \[\bar\Omega = \bigcup_{j=1}^{N}\bar{D}_j,\]

where $D_j$, $j=1,\dots , N$ are known open sets of
$\mathbb{R}^n$, satisfying the conditions below.

\begin{enumerate}
\item $D_j$, $j=1,\dots , N$ are connected and pairwise
nonoverlapping polyhedrons.

\item $\partial{D}_j$, $j=1,\dots , N$ are of Lipschitz class with
constants $r_0$, $L$.

\item There exists one region, say $D_1$, such that
$\partial{D}_1\cap\Sigma$ contains a \emph{flat} portion
$\Sigma_1$ of size $r_0$ and for every $i\in\{2,\dots , N\}$ there exists $j_1,\dots ,
j_K\in\{1,\dots , N\}$ such that

\begin{equation}\label{catena dominii}
D_{j_1}=D_1,\qquad D_{j_K}=D_i.
\end{equation}

In addition we assume that, for every $k=1,\dots , K$,
$\partial{D}_{j_k}\cap \partial{D}_{j_{k-1}}$ contains a
\emph{flat} portion $\Sigma_k$ of size $r_0$ (here we agree that
$D_{j_0}=\mathbb{R}^n\setminus\Omega$), such that

%\[\Sigma_1\subset\Sigma,\] % superflua?

\[\Sigma_k\subset\Omega,\quad\mbox{for\:every}\:k=2,\dots , K.\]

Let us emphasise that under such an assumption, for every $k=1,\dots , K$, there exists $P_k\in\Sigma_k$ and
a rigid transformation of coordinates (depending on $k$) under which we have $P_k=0$
and

\begin{eqnarray}
\Sigma_k\cap{Q}_{r_{0}/3} &=&\{x\in
Q_{r_0/3}|x_n=0\},\nonumber\\
D_{j_k}\cap {Q}_{r_{0}/3} &=&\{x\in
Q_{r_0/3}|x_n>0\},\nonumber\\
D_{j_{k-1}}\cap {Q}_{r_{0}/3} &=&\{x\in
Q_{r_0/3}|x_n<0\}.
\end{eqnarray}

\end{enumerate}
\end{enumerate}

\subsubsection{A-priori information on the potential $q$}

We shall consider a real valued function $q\in L^{\infty}(\Omega)$, with

\begin{equation}\label{apriori q}
||q||_{L^{\infty}(\Omega)}\leq E_0,
\end{equation}

for some positive constant $E_0$ and of type

\begin{subequations}
\begin{eqnarray}\label{a priori info su q}
& &q(x)=\sum_{j=1}^{N}q_{j}(x)\chi_{D_j}(x),\qquad
x\in\Omega,\label{potential 1}\\
& & q_{j}(x)=a_j+A_j\cdot x\label{potential 2},
\end{eqnarray}
\end{subequations}

where $a_j\in\mathbb{R}$, $A_j\in\mathbb{R}^n$ and $D_j$, $j=1,\dots ,
N$ are the given subdomains introduced in section \ref{subsec assumption
domain}.

%%%%%%%%%%%%%%%%%%%%%%%%  ENUNCIATO TEOREMA PRINCIPALE  %%%%%%%%%%%%%%%%%%%%%%%%%

\begin{definition}
Let $B$, $N$, $r_0$, $L$, $E_0$ be given
positive numbers with $N\in\mathbb{N}$. We will
refer to this set of numbers, along with the space dimension $n$,
as to the \textit{a-priori data}. Several constants depending on the \textit{a-priori data} will appear within the paper. In order to simplify our notation, any quantity denoted by $C,C_1,C_2, \dots$ will be called a \emph{constant} understanding in most cases that it only depends on the a priori data.
\end{definition}

\begin{remark}\label{remark finite dimensional space}
Observe that the class of functions of the form \eqref{potential 1} - \eqref{potential 2} is a finite dimensional linear space. The $L^{\infty}$ - norm $||q||_{L^{\infty}(\Omega)}$ is equivalent to the norm

\[||| q|||=\textnormal{max}_{j=1,\dots , N}\left\{|a_j|+|A_j|\right\}\]

modulo constants which only depend on the a-priori data.
\end{remark}

%From now on for simplicity we shall write

%\[\mathcal{C}_{q_i}=\mathcal{C}_i,\qquad i=1,2.\]

\begin{theorem}\label{teorema principale}
Let $\Omega$, $D_j$, $j=1,\dots , N$ and $\Sigma$ be a domain, $N$ subdomains of $\Omega$ and a portion of $\partial\Omega$ as in section \ref{subsec assumption domain} respectively.
Let $q^{(i)}$, $i=1,2$ be two potentials satisfying \eqref{apriori q} and of type

\begin{equation}\label{a priori info su sigma}
q^{(i)}=\sum_{j=1}^{N}q^{(i)}_{j}(x)\chi_{D_j}(x),\qquad
x\in\Omega,
\end{equation}

where

\[q^{(i)}_{j}(x)=a^{(i)}_j+A^{(i)}_j\cdot x,\]

with $a^{(i)}_j\in\mathbb{R}$ and $A^{(i)}_j\in\mathbb{R}^n$, then we have

\begin{equation}\label{stabilita' globale}
||q^{(1)}-q^{(2)}||_{L^{\infty}(\Omega)}\leq C d(\mathcal{C}_1,\mathcal{C}_2),
\end{equation}

where $\mathcal{C}_i$ denotes the space of Cauchy data $\mathcal{C}^{\Sigma}_{q_i}(\Sigma)$, for $i=1,2$ and $C$ is a positive constant that depends on the a-priori data
only.

\end{theorem}

\begin{remark}
An analogous result to the one obtained in theorem \ref{teorema principale} could be similarly obtained if the local Cauchy data (with vanishing condition on $\partial\Omega\setminus\overline\Sigma$ on the first component) considered here are to be replaced by the general local data with no vanishing condition on the first component as introduced in definition \ref{restrictions Cauchy data with zero trace condition}. Our present choice is motivated by the fact that the solutions that shall be actually used for the purpose of the stability estimates do indeed satisfy such zero trace condition.
\end{remark}

%\begin{remark}
%In this paper we are assuming that the conductivity $\gamma$ is a piecewise linear function. However, the case where the resistivity %function $\rho={\gamma}^{-1}$ is piecewise linear, could be treated equally well.
%\end{remark}

%%%%%%%%%%%%%%%%%%%%%%%%%%%%%%%%%%%%%%%%%%%%%%%%%%%%%%%%%%%%%%%%%%%%%%%%%%%%%%%%%

\section{Proof of the main result}\label{PMR}

The proof of our main result (theorem \ref{teorema principale}) is based on an argument that combines asymptotic type of estimates for a Green's function of the third kind for the operator

\begin{equation}\label{operatore conduttivita' misurabile}
L=\Delta + q(x)\qquad\mbox{in}\quad\Omega,
\end{equation}

(propositions \ref{ordine01}, \ref{ordine2}), with $q$ satisfying \eqref{apriori q}-\eqref{potential 2}, together with a result of unique continuation (proposition \ref{proposizione unique continuation finale}) for solutions to

\[Lu=0,\qquad\mbox{in}\quad\Omega.\]

Our idea in estimating $q^{(1)}-q^{(2)}$ exploits this singular behaviour on one hand estimating from below the blow up of  some singular solutions (we shall introduce below) $S_{{\mathcal{U}}_k}$ and some of its derivatives if  $q^{(1)}-q^{(2)}$ is large at some point. On the other hand we use estimates of propagation of smallness to show that $S_{{\mathcal{U}}_k}$ needs to be small if $q_1-q_2$ is small. We shall give the precise formulation of these results in what
follows.

\subsection{Piecewise linear potential}\label{PLP}

We collect here a series of auxiliary results. Most of the proofs are postponed to the following section.

\subsubsection{Asymptotics at interfaces}\label{subsection Green function}
We will find convenient to introduce Green's function not precisely for the physical domain $\Omega$ but for an augmented domain $\Omega_0$.

We recall that by assumption 3(c) of subsection \ref{subsec assumption domain} we can assume that there exists a point $P_1$ such that up to a rigid transformation of coordinates we have that $P_1=0$ and \eqref{flat} holds with $\Sigma=\Sigma_1$.

Denoting by

\[D_0=\left\{x\in(\mathbb{R}^n\setminus\Omega)\cap B_{r_0}\:\bigg|\:|x_i|<\frac{2}{3}r_0,\:i=1,\dots , n-1,\:\left|x_n-\frac{r_0}{6}\right|<\frac{5}{6}r_0\right\},\]

it turns out that the augmented domain $\Omega_0=\Omega\cup D_0$ is of Lipschitz class with constants $\frac{r_0}{3}$ and $\widetilde{L}$, where $\widetilde{L}$ depends on $L$ only.
Given $r>0$, we set
\begin{eqnarray}
&&\Sigma_0=\{ x\in \Omega_0 \ | \ |x_i|\le \frac{2}{3}r_0\  , \  x_n=-\frac{2}{3}r_0 \}, \\
&&(\Omega_0)_r = \{x\in \Omega_0\ | \ \mbox{dist}(x,\partial \Omega_0)>r \}.
\end{eqnarray}

In this section we shall introduce a mixed boundary value problem for $\Delta + q$ in $\Omega_0$ which is always well posed, independently of any \emph{a-priori} condition on $q$, besides the assumption of being real valued and bounded.
This shall enable us to construct a Green'n function for  $\Delta + q$ in $\Omega_0$. The underlying ideas of such construction can be traced back to an idea by Bamberger and Hua Duong \cite{Ba-Du}.

The main result is as follows.

\begin{proposition}\label{wellposedness}
Let $q\in L^{\infty}(\Omega_0)$, for any $y\in \Omega_0$ there exists a unique distributional solution $G(\cdot,y)$ to the problem

\setcounter{equation}{0}
\begin{equation}\label{Green_third_kind_statement}
\left\{
\begin{array}
{lcl}  \Delta G (\cdot, y) +q(\cdot)G(\cdot, y) =-\delta(\cdot-y)\ ,&&
\mbox{in $\Omega_0$ ,}
\\
 G(\cdot,y)= 0\ ,   && \mbox{on $\partial\Omega_0\setminus \Sigma_0$ ,}
\\
\partial _{\nu}{G}(\cdot,y) +i G(\cdot, y) =0 \ , && \mbox{on $\Sigma_0$ .}
\end{array}
\right.
\end{equation}
%Furthermore we have that $G(x,y)$ is symmetric, that is,
%\begin{equation}
%G(x,y)=G(y,x)\ \ x,y \in \Omega_0
%\end{equation}
In particular we have
\begin{eqnarray}\label{beh1}
|G(x,y)|\le C |x-y|^{2-n},\ \ && \mbox{for any}\ \ x,y\in \Omega\ , \ x\neq y \ ,
\end{eqnarray}
where $C>0$ is a constant depending on the a priori data only.
\end{proposition}
%\begin{remark}
%Quick likely, estimate \eqref{beh1} may be extended also to the case when $x$ or $y$ approach $\partial \Omega$. However, this will be not needed for our purposes, for this reason we refrain from such technicalities.
%\end{remark}

In what follows we shall make a repeated use of the solution to \eqref{Green_third_kind_statement} in the special case $q=0$.

\begin{definition}
We denote by $G_0=G_0(x,y)$ the Green's function for the Laplacian and mixed boundary conditions which solves in the distributional sense
 \begin{equation}\label{Green_third_kind_Laplace}
\left\{
\begin{array}
{lcl} \Delta G_0 (\cdot, y)=-\delta(\cdot-y)\ ,&&
\mbox{in $\Omega_0$ ,}
\\
 G_0(\cdot,y)= 0\ ,   && \mbox{on $\partial\Omega_0\setminus \Sigma_0$ ,}
\\
\partial _{\nu}{G_0}(\cdot,y) +i G_0(\cdot, y) =0 \ , && \mbox{on $\Sigma_0$ .}
\end{array}
\right.
\end{equation}
\end{definition}
\begin{remark}\label{bound}
The existence and uniqueness of a distributional solution $G_0\in L^1(\Omega_0)$ to \eqref{Green_third_kind_Laplace} is a consequence of the standard theory on boundary value problem for the Laplace equation.
It is also well-known that
\begin{equation}
G_0 (\cdot , y)\in H^1(\Omega_0\setminus B_{\epsilon}(y)),\ \ \mbox{for every}\ \epsilon >0
\end{equation}
and
\begin{eqnarray}\label{behint}
|G_0(x,y)|\le C |x-y|^{2-n},\ \ \mbox{for any}\ \ \ x,y\in ({\Omega_0})_{\frac{r_0}{4}}\ ,\ x\neq y,
%&&\mbox{dist}(x,\partial \Omega_0)\ge \frac{r_0}{4}\ , \ \mbox{dist}(y,\partial \Omega_0)\ge \frac{r_0}{4}\nonumber
\end{eqnarray}
where $C>0$ is a constant depending on the a priori data only.

Estimate \eqref{behint} holds also when $x$ and $y$ approach $\partial \Omega$. Indeed, by adapting the change of variable arguments  in Lemma 4.5 in \cite{S}, we may consider an homogeneous Neumann condition on $\Sigma_0$ instead. The latter, together with the techniques carried over in \cite[Prop. 8.3]{A-S}, enables to perform an even extension $\widetilde{G_0}$ of $G_0$ across $\Sigma_0$ . We define the set $\Sigma_0^{\perp}=\{ x\in \mathbb{R}^n\ | \ |x_i|=\frac{2}{3}r_0\ , \  -\frac{19}{24}r_0\le x_n\le -\frac{13}{24}r_0\}$ and we observe that the extension $\widetilde{G_0}$ satisfies an homogeneous Dirichlet condition on $\Sigma_0^{\perp}$ and it is Lipschitz continuous up to $\Sigma_0^{\perp}$ . Finally, the results in \cite{Lit-St-W} allow us to conclude that
\begin{eqnarray}\label{beh}
|G_0(x,y)|\le C |x-y|^{2-n},\   \mbox{for any}\ \ \ x,y\in {\Omega_0},\ x\neq y,
%&&\mbox{dist}(x,\partial \Omega_0)\ge \frac{r_0}{4}\ , \ \mbox{dist}(y,\partial \Omega_0)\ge \frac{r_0}{4}\nonumber
\end{eqnarray}
where $C>0$ is a constant depending on the a priori data only.
%As a consequence, for every $y\in \Omega_0$ such that $y\in ({\Omega_0})_{\frac{r_0}{4}}$
%\begin{equation}
%\int_{\Omega_0}|G_0(x,y)|^2 \mbox{d}x\le C
%\end{equation}
%with $C>0$ only depending on the a priori data.
\end{remark}

\begin{proposition}
For any $y\in \Omega_0$ and every $r>0$ we have that
\begin{equation}\label{gradientG}
\int_{\Omega_0\setminus B_r(y)}|\nabla G(\cdot,y)|^2\le C r^{2-n},
\end{equation}
\end{proposition}
where $C>0$ is a constant depending on $n, r_0, L$ and on $\|q\|_{L^{\infty}(\Omega_0)}$.
\begin{proof}
The proof can be obtained in a straightforward fashion by combining Caccioppoli inequality with \eqref{beh1}.
\end{proof}

\begin{proposition}\label{ordine01}
For every $x,y\in (\Omega_0)_{r_0}, \ x\neq y$ we have that
\begin{equation}\label{3a}
|G(x,y)-\Gamma(x,y)|\le \left\{
\begin{array}
{lcl} \ {C} \ , &&\mbox{if}\ n=3,\\
\ {C} (|\log|x-y|| +1)\ , &&\mbox{if}\ n=4,\ \\
\ {C} |x-y|^{4-n},
&&\mbox{if}\ n\ge 5 ,
\end{array}
\right.
\end{equation}
and
\begin{equation}\label{3b}
|\nabla_y G(x,y)-\nabla_y \Gamma(x,y)|\le \left\{
\begin{array}
{lcl}
\ {C} (|\log|x-y|| +1)\ , &&\mbox{if}\ n=3,\ \\
\ {C} |x-y|^{3-n},
&&\mbox{if}\ n\ge 4 .
\end{array}
\right.
\end{equation}
\end{proposition}

Our next goal is to analyze the asymptotics of $\nabla^2_y G(x,y)$ ($\nabla^2_y$= Hessian matrix). To this purpose it is necessary to use accurately the fine structure of $q$ as a piecewise linear function. We are especially interested to the case when $x,y$ are near to an interface of $q$. More precisely, we examine the case when $x$ and $y$ are on the opposite sides of the interface and $y$ approaches orthogonally the interface.

\begin{proposition}\label{ordine2}
Let $Q_{{l}+1}$ be a point such that  $Q_{{l}+1}\in B_{\frac{r_0}{8}}(P_{{l}+1})\cap \Sigma_{{l}+1}$ with $l\in\{1, \dots, N-1 \}$\ .
There exist a constant $\overline{C}>0$ depending on $n, r_0, L$ and $\|q\|_{L^{\infty}(\Omega_0)}$ such that following inequality hold true for every ${x}\in B_{\frac{r_0}{16}}(Q_{l+1})\cap
D_{j_{l+1}}$ and every $=Q_{l+1}-r e_n$, where $r\in
(0,\frac{r_0}{{16}})$
\begin{eqnarray}\label{second_order_statement}
|\nabla^2_y (G(x,y)-\Gamma(x,y))|\le \overline{C}r^{2-n}\  .
\end{eqnarray}
\end{proposition}

%\begin{remark}
%\begin{equation}\label{remark on SU}
%|\partial_{y_n}\partial_{z_n}S_{\mathcal{U}}({y},z)|\leq C
%||\sigma^{(1)}-\sigma^{(2)}||_{L^{\infty}(\Omega)}\left(d(y)d(z)\right)^{-\frac{n}{2}},
%\quad\mbox{for\:every}\:y,z\in\mathcal{W},
%\end{equation}

%where $d(y)=dist(y,\mathcal{U})$ and $C$ is a positive constant
%depending on $\lambda, \bar{\gamma}$ and $n$ only.
%\end{remark}

%Observe that \eqref{remark on SU} is a straightforward consequence
%of H\"older inequality and Proposition \ref{proposizione green
%function}. We constructed in this way an integral function
%$S_{\mathcal{U}}(\cdot,\cdot)$ on $\mathcal{W}\times\mathcal{W}$,
%which is written in terms of the two Green's functions $G_1(\cdot,
%y)$, $G_2(\cdot, z)$ of $L_1$, $L_2$ respectively;
%$S_{\mathcal{U}}(\cdot,z)$, $S_{\mathcal{U}}(y,\cdot)$ are in turn
%solutions for $L_1$, $L_2$ respectively on the complement part of
%$\mathcal{U}$ in $\Omega$. More precisely we have

%\end{remark}

%\[L=\mbox{div}\left(\gamma\nabla\right)\qquad\mbox{in}\:\Omega.\]

\subsubsection{Quantitative unique continuation}
We consider the operator  $L_i$ given by
\begin{equation}\label{operatori Li}
L_i =\Delta +\widetilde{q}^{(i)}(x)\qquad\mbox{in}\quad\Omega_0,\quad
i=1,2,
\end{equation}
where $\widetilde{q}^{(i)}$ is the extension on $\Omega_0$ of $q^{(i)}$ obtained by setting $\widetilde{q}^{(i)}|_{D_0}=1$,  for $i=1,2$.
For every $y\in \Omega_0$, we shall denote with $G_i(\cdot, y)$ the Green function solution to \eqref{Green_third_kind_statement} when $q=\widetilde{q}^{(i)}$ for $i=1,2$. We also define the set

%%%%%%%%%%%%%%%%%%%                     DEFINITION K_0                   %%%%%%%%%%%%%%%%%%%

\[\widehat{D}_{0}=\left\{x\in D_0\:|\:\textnormal{dist}(x,\Sigma_1)\geq \frac{r_0}{3}\right\}.\]

Let $K\in{1,\dots , N}$ be such that

\begin{eqnarray}\label{DK max}
E=\|q^{(1)}-q^{(2)}\|_{L^{\infty}(\Omega)}=\|q^{(1)}-q^{(2)}\|_{L^{\infty}(D_K)}
\end{eqnarray}

%\begin{eqnarray}\label{DK max}
%E=\|\gamma^{(1)}-\gamma^{(2)}\|_{L^{\infty}(\Omega)}=\|\gamma^{(1)}-\gamma^{(2)}\|_{L^{\infty}(D_K)}.
%\end{eqnarray}

and recall that there exist $j_1,\dots , j_K\in{1,\dots , N}$ such that

\[D_{j_1}=D_1,\dots D_{j_K}=D_K,\]

with $D_{j_1},\dots D_{j_K}$ satisfying assumption $3(c)$. For simplicity, let us rearrange the indices of these subdomains so that the above mentioned chain is simply denoted by $D_1, \dots, D_K,\ K\le N$.  We also denote
\begin{eqnarray}
\mathcal{W}_k = \bigcup_{i=0}^{k}D_{i},\qquad \mathcal{U}_k = \Omega_0\setminus\overline{\mathcal{W}_k},\quad\textnormal{for}\:k=1,\dots,K
\end{eqnarray}
and for any $y,z\in \mathcal{W}_k$
\begin{eqnarray}
{\tilde{S}}_{\mathcal{U}_k}(y,z)=\int_{\mathcal{U}_k}(\widetilde{q}^{(1)}-\widetilde{q}^{(2)})\widetilde{ G}_1(\cdot,y)\widetilde{G}_2(\cdot,z),\quad \textnormal{for}\:k=1,\dots,K.
\end{eqnarray}

It is a relatively straightforward matter to see that for every $y,z\in \mathcal{W}_k$ with $k=0,\dots, K $ we have that ${\widetilde{S}}_{{\mathcal{U}}_k}(\cdot,z),
{\widetilde{S}}_{{\mathcal{{U}}}_k}(y,\cdot)\in H^1_{loc}({\mathcal{W}}_k)$ are weak solutions to
\begin{eqnarray}
& &\left(\Delta+q^{(1)}(x)\right)
{\widetilde{S}}_{{{\mathcal{U}}}_k}(\cdot,z)=0,\quad\mbox{in}\:{\mathcal{W}}_k ,\\
& &\left(\Delta+q^{(2)}(x)\right){\widetilde S}_{{{\mathcal{U}}}_k}(y,\cdot)=0
,\quad\mbox{in}\:{\mathcal{W}}_k.
\end{eqnarray}

It is expected that some derivatives of ${\widetilde {S}}_{{\mathcal{{U}}}_k}(y,z)$ blow up as $y,z$ approach simultaneously one point of $\partial{\mathcal{U}}_k$. The following estimate for $\widetilde{S}_{\mathcal{U}_k}(y,z)$ holds true, for $k=1,\dots , K$.
%We denote by $\mathcal{W}$ a connected subset of $\mathcal{W}_K$ with Lipschitz boundary and such that $\overline{\mathcal{W}}\cap\partial D_j=\Sigma_j\cup\Sigma_{j+1}$, for $j=1,2,\dots , K$, $\widehat{D}_0\subset \mathcal{W}$ %and $\textnormal{dist}(\mathcal{W},\partial\mathcal{W}_K\setminus(\Sigma_{1}\cup\Sigma_{K+1}))>\frac{r_0}{16}$.

\begin{proposition}\label{proposizione unique continuation finale}({\bf{Estimates of unique continuation}})
If, for a positive number $\varepsilon_0$, we have

\begin{equation}\label{estim0}
\left|\widetilde{S}_{\mathcal{U}_k}(y,z)\right|\leq
r_0\varepsilon_0,\quad for\:every\: (y,z)\in
\widehat{D}_0\times \widehat{D}_0,
\end{equation}

%\begin{equation}\label{ip S tilde}
%\left|\partial_{z_i}\tilde{S}_{\mathcal{U}_K}(y,z)\right|\leq
%r_0^{1-n}\varepsilon_0,\quad for\:every\: (y,z)\in
%(D_0)_{\frac{r_0}{4}}\times(D_0)_{\frac{r_0}{4}},
%\end{equation}

then the following inequalities hold true for every $r\leq 2r_1$

\begin{equation}\label{estim1}
\left|\widetilde{S}_{\mathcal{U}_k}\left(y_{k+1},y_{k+1}\right)\right|
\leq  C (E+\varepsilon_0) \left(\frac{\varepsilon_0}{E+\varepsilon_0}\right)^{r^2\beta^{2N_1}}(1+r^{2\gamma}),
\end{equation}

\begin{equation}\label{estim2}
\left|\partial_{y_j}\partial_{z_i}\widetilde{S}_{\mathcal{U}_k}\left(y_{k+1},y_{k+1}\right)\right|
\leq  C (E+\varepsilon_0)\left(\frac{\varepsilon_0}{E+\varepsilon_0}\right)^{r^2\beta^{2N_1}}(1+r^{2\gamma})r^{-2},
\end{equation}

\begin{equation}\label{estim3}
\left|\partial^{2}_{y_j}\partial^{2}_{z_i}\widetilde{S}_{\mathcal{U}_k}\left(y_{k+1},y_{k+1}\right)\right|
\leq  C (E+\varepsilon_0)\left(\frac{\varepsilon_0}{E+\varepsilon_0}\right)^{r^2\beta^{2N_1}}(1+r^{2\gamma})r^{-4},
\end{equation}

for any $i,j=1,\dots , n$, where $y_{k+1}=P_{k+1}-2r\nu(P_{k+1})$, $\nu$ is the exterior unit normal to $\partial{D}_k$ at $P_{k+1}$, $\beta=\frac{\ln(8/7)}{\ln 4}$, $r_1=\frac{r_0}{16}$, $\gamma=2-\frac{n}{2}$,  and the constants $N_1, C>0$ depend on the a-priori data only.

\end{proposition}

Note that $\gamma<0$ only when $n>4$, hence the term $r^{2\gamma}$ becomes irrelevant if $n=3,4$.

%%%%%%%%%%%%%%%%%%%%%%%%%%%%%%%%%%%%%%%%%%%%%%%%%%%%%%%%%%%%%%%%%%%%%%%%%%%%%%%%%%%%%%%%%%%%%%%%%%%%%%%%%%%%%%%%%%%%%%%%%%%%%%%%%

\subsection{Lipschitz stability for piecewise linear potentials}\label{LSPLP}

\begin{proof}[\textit{Proof of theorem \ref{teorema principale}}]
Let $D_K$ be the subdomain of $\Omega$ satisfying \eqref{DK max} and let $D_1,\dots D_K$ be the chain of domains satisfying assumption $3(c)$. For any $k=1,\dots , K$ we shall denote by $D_{T}f$ and $\partial_\nu f$ the $n-1$ dimensional vector of the tangential partial derivatives of a function $f$ on $\Sigma_k$ and the normal partial derivative of $f$ on $\Sigma_k$ respectively. We denote by $C_i$ the local Cauchy data, for $i=1,2$. Let us also simplify our notation by replacing $C_{q^{(i)}}^{\Gamma}$ with  We shall also denote

\[\varepsilon_0=d(C_1,C_2),\qquad \delta_l=||\tilde{q}^{(1)}-\tilde{q}^{(2)}||_{L^{\infty}(\mathcal{W}_l)}\]

and introduce for any number $b>0$, the concave non decreasing function $\omega_{b}(t)$, defined on $(0,+\infty)$,

\begin{equation}\label{definition omega}
\omega_{b}(t)=\left\{ \begin{array}{ll} 2^{b}e^{-2}|\log t|^{-b},
&\quad
t\in (0,e^{-2}),\\
e^{-2}, &\quad t\in[e^{-2},+\infty).
\end{array} \right.
\end{equation}

We recall (see (4.34) and (4.35) in \cite{A-V}) that for any $\beta\in(0,1)$ we have that

\begin{eqnarray}
(0,+\infty)\ni t \rightarrow \ t\omega_{b}\left(\frac{1}{t}\right)\ \ \mbox{is a nondecreasing function}
\end{eqnarray}
and
\begin{eqnarray}
\omega_{b}\left(\frac{t}{\beta}\right)\le |\log e\beta^{-1/2}|^b\omega_{b}(t)\ , \ \ \omega_{b}(t^{\beta})\le \left(\frac{1}{\beta} \right)^b\omega_{b}(t)\ .
\end{eqnarray}

Furthermore, we set $\omega_{\alpha}^{(0)}(t)=t^{\alpha}$ with $0<\alpha<1$
and we shall denote the iterated compositions
\begin{eqnarray}
\omega_{b}^{(1)}=\omega_b\ ,\qquad \omega_{b}^{(j)}=\omega_{b}\circ \omega_{b}^{(j-1)},\ \ j=2,3, \dots.
\end{eqnarray}

We begin by noticing that for each $l=1,2,\dots$ $||\widetilde{q}^{(1)}_l - \widetilde{q}^{(2)}_l ||_{L^{\infty}(D_l)}$ can be evaluated in terms of the quantities

\begin{eqnarray}
& & ||\widetilde{q}^{(1)}_l - \widetilde{q}^{(2)}_l ||_{L^{\infty}(\Sigma_l\cap B_{\frac{r_0}{4}}(P_l))},\label{gamma interfaccia}\\
& &\left| \partial_{\nu}(\widetilde{q}^{(1)}_l - \widetilde{q}^{(2)}_l)(P_l)\right|,\label{gamma normale interfaccia}
\end{eqnarray}

where $r_0>0$ is the constant introduced in subsection \ref{subsec notation and definitions}. In fact, let us denote

\begin{equation}\label{gamma linear on Dl}
\alpha_l+\beta_l\cdot x = (\widetilde{q}^{(1)}_l -\widetilde{q}^{(2)}_l)(x),\qquad x\in D_l
\end{equation}

and choose $\{e_j\}_{j=1,\dots, n-1}$ orthonormal vectors starting at $P_l$ and generating the hyperplane containing the flat part of $\Sigma_l$.
By computing $\widetilde{q}^{(1)}_l -\widetilde{q}^{(2)}_l$ on the points $P_l$, $P_l+\frac{r_0}{5}e_j$, $j=1,\dots, n-1$ and taking their differences we obtain

\begin{equation}
|\alpha_l+\beta_l\cdot P_l|+\sum_{j=1}^{n-1}|\beta_l\cdot e_j| \leq  C||\widetilde{q}^{(1)}_l -\widetilde{q}^{(2)}_l||_{L^{\infty}(\Sigma_l\cap B_{\frac{r_0}{4}}(P_l))}.
\end{equation}

Next we notice that

\begin{equation}
|\beta_l\cdot\nu|=\left|\partial_{\nu}(q^{(1)}_l-q^{(2)}_l)(P_l)\right|.
\end{equation}

Hence each of the components of $\beta_l$ can be estimated and eventually also $|\alpha_l|$. In conclusion

\begin{equation}
|\alpha_l|+|\beta_l|\leq C\left(||\widetilde{q}^{(1)}_l -\widetilde{q}^{(2)}_l||_{L^{\infty}(\Sigma_l\cap B_{\frac{r_0}{4}}(P_l))} + \left|\partial_{\nu}(q^{(1)}_l-q^{(2)}_l)(P_l)\right|\right).
\end{equation}

Hence our task will be to estimate the quantities

\[||\widetilde{q}^{(1)}_l -\widetilde{q}^{(2)}_l||_{L^{\infty}(\Sigma_l\cap B_{\frac{r_0}{4}}(P_l))}\quad\textnormal{and}\quad\left|\partial_{\nu}(q^{(1)}_l-q^{(2)}_l)(P_l)\right|\]

iteratively with respect to $l$ (see for example \cite{A-dH-G-S}). When $l=1$ this corresponds to a stability estimate at the boundary for the potential $q$ and its normal derivative of the following type

\begin{equation}\label{estimate appendix}
||\widetilde{q}^{(1)}_1 -\widetilde{q}^{(2)}_1||_{L^{\infty}(\Sigma_l\cap B_{\frac{r_0}{4}}(P_1))}+\left|\partial_{\nu}(q^{(1)}_1-q^{(2)}_l)(P_1)\right|\leq  C(\varepsilon_0+E)\left(\frac{\varepsilon_0}{\varepsilon_0+E}\right)^{\eta_1},
\end{equation}

for some constant $\eta_1$, $0<\eta_1<1$, resulting in

\[\delta_1 \leq C(\varepsilon_0+E)\left(\frac{\varepsilon_0}{\varepsilon_0+E}\right)^{\eta_1},\quad 0<\eta_1<1.\]

The bound \eqref{estimate appendix} is proven in subsection \ref{stability at the boundary}. We proceed to estimate $\delta_2$ by proving

\begin{equation}\label{estimate on D2}
||\widetilde{q}^{(1)}_2 -\widetilde{q}^{(2)}_2||_{L^{\infty}(\Sigma_l\cap B_{\frac{r_0}{4}}(P_2))}+\left|\partial_{\nu}(q^{(1)}_2-q^{(2)}_2)(P_2)\right|\leq  C(\varepsilon_0+E)\omega_{\eta_2}\left(\frac{\varepsilon_0}{\varepsilon_0+E}\right),
\end{equation}

for some constant $\eta_2$, $0<\eta_2<1$, where $\omega_b$ is the function defined in \eqref{definition omega}. We recall that for every $y,z\in D_0$ we have

\begin{eqnarray}\label{Alessandrini 1}
& &\int_{\partial\Omega}\left(\widetilde{G}_1(x,y)\partial_{\nu}\widetilde{G}_2(x,z)-\widetilde{G}_2(x,z)\partial_{\nu}\widetilde{G}_1(x,y)\right)\:dS(x)\\
& & = \widetilde{S}_{\mathcal{U}_{1}}(y,z)\:dx + \int_{\mathcal{W}_1}(\widetilde{q}^{(1)}-\widetilde{q}^{(2)})(x)\widetilde{G}_1(x,y)\widetilde{G}_2(x,z)\:dx,\nonumber
\end{eqnarray}

and for every $i=1,\dots n$

\begin{eqnarray}\label{Alessandrini 2}
& &\int_{\partial\Omega}\left(\partial_{y_i}\widetilde{G}_1(x,y)\partial_{\nu}\partial_{z_i}\widetilde{G}_2(x,z)-\partial_{z_i}\widetilde{G}_2(x,z)\partial_{\nu}\partial_{y_i}\widetilde{G}_1(x,y)\right)\:dS(x)\\
& & = \partial_{y_i}\partial_{z_i}\widetilde{S}_{\mathcal{U}_{1}}(y,z) + \int_{\mathcal{W}_1}(\widetilde{q}^{(1)}-\widetilde{q}^{(2)})(x)\partial_{y_i}\widetilde{G}_1(x,y)\partial_{z_i}\widetilde{G}_2(x,z)\:dx,\nonumber
\end{eqnarray}

where $dS(x)$ denotes surface integration with respect to the variable $x$. Here an argument of propagation of smallness allows to get estimates on $\Sigma_2$. In order to estimate $ \widetilde{q}^{(1)}-\widetilde{q}^{(2)}$ we can repeat the argument already used in \cite{Beretta2012} to prove \eqref{Alessandrini 2} and obtain (omitting the details)

\begin{equation}\label{lipschitz stability gamma boundary 1}
|| \widetilde{q}_2^{(1)}-\widetilde{q}_2^{(2)}||_{L^{\infty}\left(\Sigma_2\cap B_{\frac{r_0}{4}}(P_2)\right)}\leq C(\varepsilon_0+E)\left|\ln\left(\frac{\varepsilon_0}{\varepsilon_0+E}\right)\right|^{-\frac{1}{4}},
\end{equation}

where $r_0>0$ is the constant introduced in subsection \ref{subsec notation and definitions}. In comparison to \cite{Beretta2012}, we have the additional task to estimate

\[\left|\partial_{\nu}(q^{(1)}_2-q^{(2)}_2)(P_2)\right|.\]

In this case we detail the procedure. Let us consider for any $i,j=1,\dots n$

\begin{eqnarray}\label{Alessandrini 3}
& &\int_{\partial\Omega} \!\!\!\Big(\partial^{2}_{y_i y_j}\widetilde{G}_1(x,y)\partial_{\nu}\partial^{2}_{z_i z_j}\widetilde{G}_2(x,z)-\partial^{2}_{z_i z_j}\widetilde{G}_2(x,z)\partial_{\nu}\partial^{2}_{y_i y_j}\widetilde{G}_1(x,y)\Big)\:dS(x)\\
& &= \partial^{2}_{y_i y_j}\partial^{2}_{z_i z_j}\widetilde{S}_{\mathcal{U}_{1}}(y,z) \nonumber\\
& &+\int_{\mathcal{W}_1}(\widetilde{q}^{(1)}-\widetilde{q}^{(2)})(x)\partial^{2}_{y_i y_j}\widetilde{G}_1(x,y)\cdot\partial^{2}_{z_i z_j}\widetilde{G}_2(x,z)\:dx.\nonumber
\end{eqnarray}

From \eqref{Alessandrini 1} we obtain

\begin{eqnarray}
|\widetilde{S}_{\mathcal{U}_1}(y,z)| &\leq & \varepsilon_0
\left(\left|\left|\widetilde{G}_1(\cdot,y)\right|\right|_{H^{\frac{1}{2}}_{00}(\Sigma)}+\left|\left|\partial_{\nu}\widetilde{G}_2(\cdot,z)\right|\right|_{H^{-\frac{1}{2}}(\partial\Omega)\big\vert_{\Sigma}}\right)\nonumber\\
& \times &\left(\left|\left|\widetilde{G}_2(\cdot,z)\right|\right|_{H^{\frac{1}{2}}_{00}(\Sigma)}+\left|\left|\partial_{\nu}\widetilde{G}_1(\cdot,y)\right|\right|_{H^{-\frac{1}{2}}(\partial\Omega)\big\vert_{\Sigma}}\right)\nonumber\\
& + & \delta_1 \left|\left|\widetilde{G}_1(\cdot,y)\right|\right|_{L^{2}(\mathcal{W}_1)}\left|\left|\widetilde{G}_2(\cdot,z)\right|\right|_{L^{2}(\mathcal{W}_1)}\nonumber\\
&\leq & C \left(\varepsilon_0+\delta_1\right) r_0^{2-n},\qquad\textnormal{for\:every}\:y,z\in\widehat{D}_0.
\end{eqnarray}

Let $\rho_0=\frac{r_0}{\overline{C}}$, where $\overline{C}$ is the constant introduced in proposition \ref{ordine2}, let $r\in(0,2r_1)$ (where $r_1$ was introduced in proposition \ref{proposizione unique continuation finale}) and denote

\[y_2=P_2+r\nu ,\]

then for every $i,j=1,\dots ,n$

\begin{equation}\label{S=I1+I2}
\partial^{2}_{y_i y_j}\partial^{2}_{z_i z_j}\widetilde{S}_{\mathcal{U}_1}(y_2,y_2)=I_1^{ij}(y_2)+I_2^{ij}(y_2),
\end{equation}

where

\[I_1^{ij}(y_2)=\int_{B_{\rho_0}(P_2)\cap D_2}(q^{(1)}-q^{(2)})(\cdot)\partial^{2}_{y_i y_j}\widetilde{G}_1(\cdot,y_2)
\partial^{2}_{z_i z_j}\widetilde{G}_2(\cdot,y_2),\]

\[I_2^{ij}(y_2)=\int_{\Omega\setminus (B_{\rho_0}(P_2)\cap
D_2)}(q^{(1)}-q^{(2)})(\cdot)\partial^{2}_{y_i y_j}\widetilde{G}_1(\cdot,y_2)
\partial^{2}_{z_i z_j}\widetilde{G}_2(\cdot,y_2).\]

If for $k=1,2$ we denote by $|I_k(y_2)|$ the Euclidean norm of matrix $I_k(y_2)=\{I_k^{ij}(y_2)\}_{i,j=1,\dots , n}$, we have

\begin{equation}\label{stima I2}
|I_2(y_2)|\leq CE\rho_0^{-n},
\end{equation}

where $C$ depends on $\lambda$ and $n$ only (see \cite{A-dH-G-S}) and

\begin{eqnarray*}
 & &|I_1(y_2)|\nonumber\\
 & & \geq \frac{1}{n}\sum_{i,j=1}^{n}\Bigg\{\left|\int_{B_{\rho_0}(P_2)\cap D_2}(\partial_{\nu}(q^{(1)}_2-q^{(2)}_2)(P_2))(x-P_2)_n\partial^{2}_{y_i y_j}\widetilde{G}_1(x,y_2)
\partial^{2}_{z_i z_j}\widetilde{G}_2(x,y_2)\:dx\:\right|\nonumber\\
& &-\int_{B_{\rho_0}(P_2)\cap D_2}|(D_T(q^{(1)}_2-q^{(2)}_2)(P_2))\cdot (x-P_2)'||\partial^{2}_{y_i y_j}\tilde{G}_1(x,y_2)|\:|
\partial^{2}_{z_i z_j}\widetilde{G}_2(x,y_2)\:dx\:|\nonumber\\
& &-\int_{B_{\rho_0}(P_2)\cap D_2}|(q^{(1)}_2-q^{(2)}_2)(P_2)||\partial^{2}_{y_i y_j}\widetilde{G}_1(x,y_2)|\:|
\partial^{2}_{z_i z_j}\widetilde{G}_2(x,y_2)\:dx\:|\Bigg\}
\end{eqnarray*}

and, noticing that up to a transformation of coordinates we can assume that $P_2$ coincides with the origin $O$ of the coordinates system and by theorem \ref{ordine2}, this leads to

\begin{eqnarray}\label{stima S}
|I_1(y_2)|
& &\geq C\bigg\{
|\partial_{\nu}(q^{(1)}_2-q^{(2)}_2) (O)|\int_{B_{\rho_0}(O)\cap
D_2}|\nabla^{2}_y\Gamma(x,y_2)|^2 |x_n|\:dx\nonumber\\
& &- E\int_{B_{\rho_0}(O)\cap
D_2}|\nabla^{2}_y\Gamma(x,y_2)||x-y_2|^{2-n}|x_n|\:dx\nonumber\\
& &-E\int_{B_{\rho_0}(O)\cap
D_2}|x-y_2|^{4-2n}|x_n|\:dx\bigg\}\nonumber\\
& &-\int_{B_{\rho_0}(O)\cap D_2}|D_T(q^{(1)}_2-q^{(2)}_2)|\:|x'|\:|\nabla^{2}_y\widetilde{G}_1(x,y_2)|\:|
\nabla^{2}_z\widetilde{G}_2(x,y_2)|\:dx\nonumber\\
& & -\int_{B_{\rho_0}(O)\cap D_2}|(q^{(1)}_2-q^{(2)}_2)(O)|\:|\nabla^{2}_y\widetilde{G}_1(x,y_2)|\:|
\nabla^{2}_z\widetilde{G}_2(x,y_2)|\:dx.
\end{eqnarray}

Therefore, by combining \eqref{stima S} together with \eqref{S=I1+I2} and \eqref{stima
I2}, we obtain

\begin{eqnarray*}
|I_1(y_2)| &\geq & C\bigg\{
|\partial_{\nu}(q^{(1)}_2-q^{(2)}_2)(O)|\int_{B_{\rho_0}(O)\cap
D_2}|x-y_2|^{-2n}\:|x_n|\:dx\noindent\\
&-&2E\int_{B_{\rho_0}(O)\cap
D_2}|x-y_2|^{3-2n}\:dx\noindent\\
&-&(\varepsilon_0+E)\left|\ln\left(\frac{\varepsilon_0}{\varepsilon_0+E}\right)\right|^{-\frac{1}{4}}\int_{B_{\rho_0}(O)\cap D_2}\:|x-y_2|^{1-2n}\:dx\nonumber\\
&-&(\varepsilon_0+E)\left|\ln\left(\frac{\varepsilon_0}{\varepsilon_0+E}\right)\right|^{-\frac{1}{4}}\int_{B_{\rho_0}(O)\cap D_2}\:|x-y_2|^{-2n}\:dx\bigg\},
\end{eqnarray*}

which leads to

\begin{eqnarray}
|\partial_{\nu}(q^{(1)}_2-q^{(2)}_2)|r^{1-n} &\le & |I_1(y_2)| + C\bigg((\varepsilon_0+E)\left|\ln\left(\frac{\varepsilon_0}{\varepsilon_0+E}\right)\right|^{-\frac{1}{4}} r^{-n}\nonumber\\
& + & 2Er^{3-n}\bigg),
\end{eqnarray}

where

\begin{equation}\label{I1}
|I_1(y_2)|\le |\partial^2_{y_n}\partial^2_{z_n}\widetilde{S}_{\mathcal{U}_{1}}(y_2,y_2)| + C E \rho_0^{-n}.
\end{equation}

Thus by combining the last two inequalities we get

\begin{eqnarray}
|\partial_{\nu}(q^{(1)}_2-q^{(2)}_2)|r^{1-n} &\le & |\partial^2_{y_n}\partial^2_{z_n}\tilde{S}_{\mathcal{U}_{1}}(y_2,y_2)|+ C\bigg\{E \left(\rho_0^{-n}+ r^{3-n}\right)\nonumber\\
&+& (\varepsilon_0+E)\left|\ln\left(\frac{\varepsilon_0}{E+\varepsilon_0+\delta_1}\right)\right|^{-\frac{1}{4}}r^{-n}\bigg\}
\end{eqnarray}

and by recalling that by proposition \ref{proposizione unique continuation finale} we have

\[
\left|\partial^{2}_{y_j}\partial^{2}_{z_i}\tilde{S}_{\mathcal{U}_1}\left(y_{2},y_{2}\right)\right|
\leq C\left(\frac{\varepsilon_0 +\delta_1}{E+\varepsilon_0 +\delta_1}\right)^{r^2\beta^{2N_1}}(\varepsilon_0+E)\left(1+r^{4-n}\right)r^{-4},
\]

where $\beta, N_1$ are the constants introduced in proposition \ref{proposizione unique continuation finale}, we obtain

\begin{eqnarray}\label{stima pre minimizzazione}
|\partial_{\nu}(q^{(1)}_2-q^{(2)}_2)| &\le & C \bigg\{\left(\frac{\varepsilon_0+\delta_1}{E+\varepsilon_0+\delta_1}\right)^{r^2\beta^{2N_1}}(\varepsilon_0+E+\delta_1)r^{p_n}+E r^2\nonumber\\
&+ & (\varepsilon_0+E)\left|\ln\left(\frac{\varepsilon_0}{\varepsilon_0+E}\right)\right|^{-\frac{1}{4}}r^{-1}\bigg\},
\end{eqnarray}

for any $r$, $0<r<2r_1$, where $r_1$ is as above and $p_n= \left\{\begin{array}{cc} -1 & n\geq 4\nonumber\\ n-5 & n=3\end{array}\right.$. By minimizing \eqref{stima pre minimizzazione} with respect to $r$ we obtain

\begin{equation}\label{q2 normal}
|\partial_{\nu}(q^{(1)}_2-q^{(2)}_2)|\le C (\varepsilon_0+\delta_1+E)\omega_{\eta_2}\left(\frac{\varepsilon_0+\delta_1}{\varepsilon_0+\delta_1+E}\right),\textnormal{for\:some\:}\eta_2,\quad 0<\eta_2<1
\end{equation}

and by combining \eqref{q2 normal} together with

\begin{equation}
\frac{\varepsilon_0 +\delta_1}{\varepsilon_0+\delta_1+E}\leq C \omega_{\eta_1}^{(0)} \left(\frac{\varepsilon_0}{\varepsilon_0 +E}\right)
\end{equation}

and by the properties of the modulus $\omega_b$, we obtain \eqref{estimate on D2}. By proceeding by iteration on $l$ in order to estimate $q^{(1)}_l - q^{(2)}_l$ for $l=1,\dots , K$, we replace  \eqref{Alessandrini 1}, \eqref{Alessandrini 2} and \eqref{Alessandrini 3} by

\begin{eqnarray}\label{Alessandrini 4}
& &\int_{\partial\Omega}\left(\widetilde{G}_1(x,y)\partial_{\nu}\widetilde{G}_2(x,z)-\widetilde{G}_2(x,z)\partial_{\nu}\widetilde{G}_1(x,y)\right)\:dS(x)\\
& &\widetilde{S}_{\mathcal{U}_{l-1}}(y,z)+\int_{\mathcal{W}_{l-1}}(\widetilde{q}^{(1)}-\widetilde{q}^{(2)})(x)\widetilde{G}_1(x,y)\widetilde{G}_2(x,z)\:dx,\nonumber
\end{eqnarray}

\begin{eqnarray}\label{Alessandrini 5}
& &\int_{\partial\Omega}\left(\partial_{y_i}\widetilde{G}_1(x,y)\partial_{\nu}\partial_{z_i}\widetilde{G}_2(x,z)-\partial_{z_i}\widetilde{G}_2(x,z)\partial_{\nu}\partial_{y_i}\widetilde{G}_1(x,y)\right)\:dS(x)\\
& &\partial_{y_i}\partial_{z_i}\widetilde{S}_{\mathcal{U}_{l-1}}(y,z)+\int_{\mathcal{W}_{l-1}}(\widetilde{q}^{(1)}-\widetilde{q}^{(2)})(x)\partial_{y_i}\widetilde{G}_1(x,y)\partial_{z_i}\widetilde{G}_2(x,z)\:dx\nonumber
\end{eqnarray}

for any $i=1,\dots ,n$ and

\begin{eqnarray}\label{Alessandrini 6}
& &\hspace{0.5cm}\int_{\partial\Omega}\left(\partial^{2}_{y_i y_j}\widetilde{G}_1(x,y)\partial_{\nu}\partial^{2}_{z_i z_j}\widetilde{G}_2(x,z)-\partial^{2}_{z_i z_j}\widetilde{G}_2(x,z)\partial_{\nu}\partial^{2}_{y_i y_j}\widetilde{G}_1(x,y)\right)\:dS(x)\\
& & =\partial^{2}_{y_i y_j}\partial^{2}_{z_i z_j}\widetilde{S}_{\mathcal{U}_{l-1}}(y,z)+\int_{\mathcal{W}_{l-1}}(\widetilde{q}^{(1)}-\widetilde{q}^{(2)})(x)\partial^{2}_{y_i y_j}\widetilde{G}_1(x,y)\cdot\partial^{2}_{z_i z_j}\widetilde{G}_2(x,z)\:dx\nonumber
\end{eqnarray}

for any $i,j=1,\dots , n$ respectively. Note that \eqref{Alessandrini 4} leads to

\begin{equation}
|\widetilde{S}_{\mathcal{U}_{l1}}(y,z)| \leq  C (\varepsilon_0+\delta_{l-1}) r_0^{2-n},\qquad
\textnormal{for\:every}\:y,z\in \widehat{D}_0,
\end{equation}

where $C$ depends on $L$, $\lambda$, $n$. By repeating the same argument applied for the special case $l=2$ we obtain

\begin{equation}\label{iteration 1}
|| \widetilde{q}^{(1)}_l-\widetilde{q}^{(2)}_l ||_{L^{\infty}(\Sigma_l\cap B_{\frac{r_0}{4}}(P_l))} +\left|\partial_{\nu}(q^{(1)}-q^{(2)})(P_l)\right| \leq  C(\varepsilon_0+E)\omega_{\eta_l}\left(\frac{\varepsilon_0}{\varepsilon_0+E}\right),
\end{equation}

for some $\eta_l$, $0<\eta_l<1$ and observing that

\[\delta_l\leq \delta_{l-1}+||q^{(1)}_l - q^{(2)}_l  ||_{L^{\infty}(D_l)},\]

we obtain for every $l=2,3,\dots$

\[\delta_{l}\leq \delta_{l-1}+C(\varepsilon_0+\delta_{l-1}+E)\omega_{\eta_l}\left(\frac{\varepsilon_0+\delta_{l-1}}{\varepsilon_0+\delta_{l-1}+E}\right),
\]

hence trivially

\begin{equation}\label{iteration 2}
\frac{\varepsilon_0+\delta_l}{\varepsilon_0+\delta_l+E}\leq C \omega_{\eta_l}\left(\frac{\varepsilon_0+\delta_{l-1}}{\varepsilon_0+\delta_{l-1}+E}\right).
\end{equation}

Using the properties of the logarithmic moduli $\omega_b$, \eqref{iteration 1} and the induction step \eqref{iteration 2} we arrive at

\begin{equation*}
||q^{(1)}-q^{(2)}||_{L^{\infty}(\Omega)}\leq C(\varepsilon_0 + E)\omega_{\eta_K}^{(K-1)}\left(\frac{\varepsilon_0}{\varepsilon_0 + E}\right),
\end{equation*}

therefore

\begin{equation}\label{463}
E\leq C(\varepsilon_0 +E)\omega_{\eta_K}^{(K-1)}\left(\frac{\varepsilon_0}{\varepsilon_0 + E}\right).
\end{equation}

Assuming that $E>\varepsilon_0 e^2$ (if this is not the case then the theorem is proven) we obtain

\[E\leq C \left(\frac{E}{e^2}+E\right)\omega_{\eta_K}^{(K-1)}\left(\frac{\varepsilon_0}{E}\right),\]

which leads to

\[\frac{1}{C}\leq \omega_{\eta_K}^{(K-1)}\left(\frac{\varepsilon_0}{E}\right),\]

therefore

\[E\leq \frac{1}{\omega_{\eta_K}^{(-(K-1))}\left(\frac{1}{C}\right)}\:\varepsilon_0,\]

where here, with a slight abuse of notation, $\omega_{\frac{1}{C}}^{(-(K-1))}$ denotes the inverse function of $\omega_{\eta_K}^{(K-1)}$.

\end{proof}

%%%%%%%%%%%%%%%%%%%%%%%%%%%%%%%%%%%%%%%%%%%%%%%%%%%%%%%%%%%%%%%%%%%%%%%%%%%

%SEZIONE DIMOSTRAZIONI STIME ASINTOTICHE, UNIQUE CONTINUATION E STABILITA' AL BORDO

%%%%%%%%%%%%%%%%%%%%%%%%%%%%%%%%%%%%%%%%%%%%%%%%%%%%%%%%%%%%%%%%%%%%%%%%%%

\section{Proof of technical propositions}\label{PP}

\subsection{Asymptotic estimates}\label{AE}

\begin{proof}[Proof of Proposition \ref{wellposedness}]

Let us consider the following inhomogeneous boundary value problem. Given $f\in L^2(\Omega_0)$, we wish to find $v\in H^1(\Omega_0)$ such that
\begin{equation}\label{Green_third_kind}
\left\{
\begin{array}
{lcl}  \Delta v+qv =f\ ,&&
\mbox{in $\Omega_0$ ,}
\\
 v= 0\ ,   && \mbox{on $\partial\Omega_0\setminus \Sigma_0$ ,}
\\
\partial _{\nu}v +iv =0 \ , && \mbox{on $\Sigma_0$ .}
\end{array}
\right.
\end{equation}
in the weak sense. Let us also consider the adjoint problem (recall that $q$ is real valued).
\begin{equation}\label{Green_third_kind_adjoint}
\left\{
\begin{array}
{lcl}  \Delta w+qw =g\ ,&&
\mbox{in $\Omega_0$ ,}
\\
 w= 0\ ,   && \mbox{on $\partial\Omega_0\setminus \Sigma_0$ ,}
\\
\partial _{\nu}w -iw =0 \ , && \mbox{on $\Sigma_0$ .}
\end{array}
\right.
\end{equation}

It is well-known \cite{G-T} that the Fredholm alternative applies and we have existence for \eqref{Green_third_kind} if and only if we have uniqueness for \eqref{Green_third_kind_adjoint} and viceversa. In fact we can prove uniqueness for either of the two. We consider the homogeneous problem
\begin{equation}\label{Green_third_kind_hom}
\left\{
\begin{array}
{lcl}  \Delta z+qz =0\ ,&&
\mbox{in $\Omega_0$ ,}
\\
 z= 0\ ,   && \mbox{on $\partial\Omega_0\setminus \Sigma_0$ ,}
\\
\partial _{\nu}z \pm iz =0 \ , && \mbox{on $\Sigma_0$ .}
\end{array}
\right.
\end{equation}
Using $\bar{z}$ as a test function we obtain
\begin{equation}
\int_{\Omega_0}|\nabla z|^2 -\int_{\Omega_0}q|z|^2 \pm i \int_{\Sigma_0} |z|^2 =0\ ,
\end{equation}
which in turn implies $z=0$ on $\Sigma_0$. Consequently also $\partial_{\nu}z=0$ on $\Sigma_0$. From the uniqueness for the Cauchy problem, it follows that $z\equiv 0$ in $\Omega_0$. Hence the solution to \eqref{Green_third_kind} exists and is unique. We now show that \eqref{Green_third_kind} is also well-posed in $H^1(\Omega_0)$.
From the weak formulation of the problem \eqref{Green_third_kind} we have
\begin{equation}\label{1a}
\int_{\Sigma_0}|v|^2= -\mbox{Im}\left(\int_{\Omega_0}f \bar{v}\right)\ ,
\end{equation}

\begin{equation}\label{2a}
\int_{\Omega_0}|\nabla v|^2= -\mbox{Re}\left(\int_{\Omega_0}f\bar{v}\right) + \int_{\Omega_0}q |v|^2\ .
\end{equation}

We define

\begin{eqnarray}
&&\varepsilon^2=\int_{\Sigma_0} |v|^2 + \int_{\Sigma_0}|\partial_{\nu} v|^2\ ,\\
&&\eta^2= \int_{\Omega_0}|f|^2\ ,\\
&&\delta^2=\int_{\Omega_0}|v|^2\ ,\\
&& E^2=\int_{\Omega_0}|\nabla v|^2 \ .
\end{eqnarray}
From \eqref{1a} and the Schwartz inequality it follows that
\begin{eqnarray}\label{1c}
\int_{\Sigma_0}|v|^2\le \eta \delta
\end{eqnarray}
which, combined with the impedance condition and  Poincar\'{e} inequality, implies that
\begin{eqnarray}\label{1d}
\varepsilon^2\le 2 \eta E.
\end{eqnarray}
Moreover, by \eqref{2a} we have that
\begin{eqnarray}\label{1e}
E^2\le \eta \delta + \|q\|_{L^{\infty}(\Omega_0)}\delta^2.
\end{eqnarray}
We claim that there exists a constant $C_1>0$ depending on $n, r_0, L$ and $\|q\|_{L^{\infty}(\Omega_0)}$ such that
\begin{eqnarray}\label{claim1}
E^2\le C_1 \eta^2 \ .
\end{eqnarray}
In order to prove our claim, we distinguish two cases. If $\delta^2\le \eta^2$, then \eqref{claim1} follows from \eqref{1e}. Otherwise, we observe that by well-known estimates for the Cauchy problem (see for instance \cite{A-R-R-V}) it follows that
\begin{eqnarray}
\delta^2\le (E^2 + \varepsilon^2 + \eta^2)\omega\left(\frac{\varepsilon^2 + \eta^2}{E^2 + \varepsilon^2 + \eta^2} \right),
\end{eqnarray}
where $\omega(t)\le C |\log (t)|^{-\mu}$ with $C$, $0<\mu<1$ depending on $n$, $r_0$, $L$ and $\|q\|_{L^{\infty}(\Omega_0)}$.
By \eqref{1d} and \eqref{1e} the above inequality leads to
\begin{eqnarray}\label{CP1}
\delta^2\le (3 \eta\delta + \|q\|_{L^{\infty}(\Omega_0)}\delta^2 +\eta^2)\omega\left(\frac{\varepsilon^2 + \eta^2}{E^2 + \varepsilon^2 + \eta^2} \right) \ .
\end{eqnarray}

From \eqref{CP1} we have that
\begin{eqnarray}
1\le (4 + \|q\|_{L^{\infty}(\Omega_0)}) \omega\left(\frac{\varepsilon^2 + \eta^2}{E^2 + \varepsilon^2 + \eta^2} \right),
\end{eqnarray}
which leads to
\begin{eqnarray}
\omega^{-1}\left(\frac{1}{4 + \|q\|_{L^{\infty}(\Omega_0)}} \right)(E^2 + \varepsilon^2 + \eta^2)\le \varepsilon^2 + \eta^2 .
\end{eqnarray}
Again from \eqref{1d} we have that
\begin{eqnarray}
\omega^{-1}\left(\frac{1}{4 + \|q\|_{L^{\infty}(\Omega_0)}} \right)E^2 \le 2\eta E+ \eta^2 .
\end{eqnarray}

By Schwartz inequality we readily obtain
\begin{eqnarray}
E^2\le C \eta^2
\end{eqnarray}
which, together with \eqref{1d}, gives
\begin{eqnarray}
\|v\|_{H^1(\Omega_0)}\le C \|f\|_{L^2(\Omega_0)}\ .
\end{eqnarray}
We set $J=[\frac{n-1}{2}]$ and we define iteratively the following kernels

\begin{equation}\label{kernels}
\left\{
\begin{array}
{lcl} R_0(x,y)=G_0(x,y)\ ,&&
\\
R_j(x,y)= \int_{{\Omega}_0} G_0(x,z)q(z)R_{j-1}(z,y)dz\ ,   &&
\end{array}
\right.
\end{equation}
for every $j=1\dots,J$. Note that for every $j=1, \dots, J$ we have
\begin{equation}\label{resti}
\left\{
\begin{array}
{lcl} \Delta R_j(\cdot, y)=-q(\cdot)R_{j-1}(\cdot, y)\ ,&&\ \ \mbox{in}\ \ \Omega_0
\\
R_j(\cdot,y)= 0\  \ ,   && \ \ \mbox{on}\ \ \partial \Omega _0\setminus \Sigma_0\\
\partial_{\nu}R_j(\cdot, y)+ i R_j(\cdot, y)=0\ \  && \ \ \mbox{on}\ \ \Sigma_0
\end{array}
\right.
\end{equation}
and also by standard estimate \cite{Mi}
\begin{equation}\label{indicibassi}
|R_j(x,y)|\le C |x-y|^{2j+2-n},\ \ \mbox{for \ every}\ j=0, \dots, J-1.
\end{equation}

Now, if $n$ is even
\begin{eqnarray}\label{even}
|R_J(x,y)|\le C \left (|\log|x-y||+1 \right),
\end{eqnarray}
where if $n$ is odd
\begin{eqnarray}\label{odd}
|R_J(x,y)|\le C \ .
\end{eqnarray}

In either case $\|R_J(\cdot, y) \|_{L^p(\Omega_0)}\le C$  for every $p<\infty$. We let $R_{J+1}(\cdot,y)=v(\cdot)$ the solution to \eqref{Green_third_kind}, when $f(\cdot)=-q(\cdot)R_{J}(\cdot,y)$. Therefore
\begin{eqnarray}
\|R_{J+1}(\cdot,y)\|_{H^1(\Omega_0)}\le C
\end{eqnarray}
and by relatively standard regularity estimates in the interior and at the boundary (see the arguments in Remark \ref{bound})
\begin{eqnarray}\label{intbound}
|R_{J+1}(x,y)|\le C , \  \ \mbox{for any}\ x, y \in \Omega_0,\  \ x\neq y \ .
\end{eqnarray}
If we form
\begin{equation}\label{somma G}
G(x,y)=G_0(x,y) + \sum_{j=1}^{J+1}R_{j}(x,y),
\end{equation}
then we end up with the desired solution to \eqref{Green_third_kind_statement}.

\end{proof}

%%%%%%%%%%%%%%%%%%%%%%%%%%%%%%%%%%%%%%%%%%%%%%%%%%%%%%%

\begin{proof}[Proof of Proposition \ref{ordine01}]
It is evident that $G_0(x,y)-\Gamma(x-y)$ is harmonic in either variables $x\in \Omega_0, y\in \Omega_0$. Hence standard interior regularity yields
\begin{eqnarray}
|G_0(x,y) - \Gamma(x-y)| + |\nabla_y(G_0(x,y) - \Gamma(x-y))|\le C ,
\end{eqnarray}
for any $x,y\in (\Omega_0)_{r_0}$.
In order to prove \eqref{3a} and \eqref{3b} it then suffices to estimate $R(x,y)= G(x,y)-G_0(x,y)$ and its $y-$derivatives.
Note that a crude estimate of $R$ could be derived from \eqref{indicibassi}, \eqref{even}, \ \eqref{odd}, \eqref{intbound} and \eqref{somma G}. Finer estimates are obtained as follows. By Green's identity we have
\begin{eqnarray}\label{Green}
R(x,y)= \int_{\Omega_0} G(x,z)q(z)G_0(z,y){d}z\ .
\end{eqnarray}
By combining \eqref{Green}
together with \eqref{beh1} and \cite[Chap. 2]{Mi} we obtain

\begin{equation}\label{stima R}
|R(x,y)|\le \left\{
\begin{array}
{lcl} \ {C} \ , &&\mbox{if}\ n=3,\\
\ {C} (|\log|x-y|| +1) \ , &&\mbox{if}\ n=4,\ \\
\ {C} |x-y|^{4-n},
&&\mbox{if}\ n\ge 5.
\end{array}
\right.
\end{equation}

We now estimate $\nabla_y R$.
When $x\neq y$ we can differentiate under the integral sign and obtain
\begin{eqnarray}
|\partial_{y_i}R(x,y)|&=&\left |\int_{\Omega_0} G(x,z)q(z)\partial_{y_i}G_0(z,y){d}z\right| \\
&\le& \int_{\Omega_0\setminus B_{r_0}(x)}|G(x,z)q(z)\partial_{y_i}G_0(z,y)|\mbox{d}z \nonumber\\
&+&C \int_{B_{r_0}(x)}|x-z|^{2-n}|z-y|^{1-n}{d}z ,\nonumber
\end{eqnarray}
hence we have used the pointwise bound on $G$ achieved in \eqref{beh1} and the above stated asymptotics on $\nabla_y G_0$. The first integral is bounded by a constant, the second one can be estimated (see \cite[Chap. 2]{Mi}) by
\begin{eqnarray}
C \left(|\log|x-y| +1 \right)\ \ \mbox{if}\ n=3
\end{eqnarray}
and by
\begin{eqnarray}
C |x-y|^{3-n}\ \ \mbox{if}\ n\ge 4 \ \ .
\end{eqnarray}

\end{proof}

%%%%%%%%%%%%%%%%%%%%%%%%%%%%%%%%%%%%%%%%%%%%%%%%%%%%%%%%%%%%%%%%%%%%%%%

\begin{proof}[Proof of Theorem \ref{ordine2}]

We fix $l\in \{1, \dots, N-1 \}$. %We set $q^-(x)=q_{j_l}(x)$ and $q^+(x)=q_{j_{l+1}}(x)$.
Furthermore, we observe that up to a transformation of coordinates we can assume that $Q_{l+1}$ coincides with the origin $0$ of the coordinates system. We denote

\[R(x,y)=G(x,y)-\Gamma(x-y),\]

then we have

\begin{equation}\label{resto}
\left\{
\begin{array}
{lcl}  \Delta_x R(x,y)+q(x) R(x,y)=-q(x)\Gamma(x-y)\ ,&&
\mbox{in $\Omega_0$ ,}
\\
 R(x,y)= -\Gamma(x,y)\ ,   && \mbox{on $\partial\Omega_0\setminus \Sigma_0$ ,}
\\
\partial _{{\nu}_x}{R}(x,y)+ i R(x,y)=-\partial_{{\nu}_x}\Gamma(x-y)- i\Gamma(x-y) , && \mbox{on $\Sigma_0$ .}
\end{array}
\right.
\end{equation}

By Green's identity we arrive at
\begin{eqnarray}
&& R(x,y)=\int_{\Sigma_0}(\partial_{{\nu}_z} + i)\Gamma(z-y)G(z,x){d}\sigma(z)\ \\
&& + \int_{\partial\Omega_0\setminus \Sigma_0}\Gamma(z-y)G(z,x)+ \int_{\Omega_0}\Gamma(z-y)q(z)G(z,x){d}z\ .\nonumber
\end{eqnarray}
We denote
\begin{eqnarray}
\widehat{R}(x,y)=\int_{B_{\frac{r_0}{8}}}\Gamma(z-y)q(z)G(z,x){d}z\ .
\end{eqnarray}
With the stated assumptions on $x$ and $y$, it is a straightforward matter to show that
\begin{eqnarray}\label{boundHessian}
|\nabla_y^2(R(x,y)-\widehat{R}(x,y))|\le C \ .
\end{eqnarray}
Let us investigate $\widehat{R}$. We set
\begin{eqnarray}
&& B=B_{\frac{r_0}{8}}\ , B'=B'_{\frac{r_0}{8}} \ ,\\
&& B^+=\{x\in B: x_n>0\},\ B^-=\{x\in B: x_n<0\} \ ,\\
&& q^+=q|_{B^+}, \ q^-=q|_{B^-} , \ [q]= (q^+-q^-)|_{B'}\ .
\end{eqnarray}
We compute
\begin{eqnarray}
\partial_{y_i}\widehat{R}(x,y) & = &- \int_B \partial_{z_i}\Gamma(z-y)q(z)G(z,x)\:dz \\
&=& -\int_{\partial B}\Gamma(z-y)q(z)G(z,x)e_i\cdot \nu\:dS(z) \nonumber\\
&&+ \int_{B'}\Gamma(z'-y)[q(z')]G(z',x)e_i\cdot e_n\:dz'\nonumber\\
&&+\int_{B}\Gamma(z-y)\left(\partial_{z_i}q(z)G(z,x)+q(z)\partial_{z_i}G(z,x) \right)\: dz. \nonumber
\end{eqnarray}

Note that $\partial_{z_i}q$ is well-defined on $B\setminus B'$. By further differentiation

\begin{eqnarray}
& &\partial^2_{y_iy_j}\widehat{R}(x,y)\\
& &= -\int_{\partial B}\partial_{y_j}\Gamma(z-y)q(z)G(z,x)\:dS(z) \nonumber\\
& &+\int_{B'}\partial_{y_j}\Gamma(z'-y)[q(z')]G(z',x)e_i\cdot e_n\:dz'  \nonumber\\
& &+\int_{B}\partial_{y_j}\Gamma(z-y)\left(\partial_{z_i}q(z)G(z,x)+q(z)\partial_{z_i}G(z,x) \right)\:dz. \nonumber
\end{eqnarray}

The first integral on the right hand side of the above equality is readily seen to be bounded. The third one is dominated by

\begin{eqnarray}
\int_{B}\frac{1}{|z-y|^{n-1}|z-x|^{n-1}}{d}z \le C |x-y|^{2-n}
\end{eqnarray}
and, since $|x-y|^2=|x_n+r|^2 + |x'|^2\ge r^2$, we can bound it as follows
\begin{eqnarray}
\int_{B}\frac{1}{|z-y|^2|z-x|^2}{d}z \le C r^{2-n}\ .
\end{eqnarray}
The second integral, the one on $B'$, is nontrivial only when $i=j=n$. Therefore,
\begin{eqnarray}\label{ij}
|\partial^2_{y_iy_j}R(x,y)|\le C|x-y|^{2-n} \ \ \mbox{for all}\ \ (i,j)\neq (n,n)\ .
\end{eqnarray}
Now
\begin{eqnarray}
\partial^2_{y_n}R(x,y)= \Delta_{y}R(x,y) - \Delta'_{y}R(x,y),
\end{eqnarray}
where $\Delta'_{y}R(x,y)= \sum_{i=1}^{n-1}\partial^2_{y_i}R(x,y)$. By the symmetry of $G$, we also have $R(x,y)=R(y,x)$, hence
\begin{eqnarray}
\Delta_{y}R(x,y)+q(y)R(x,y)=-q(y)\Gamma(x-y)\ .
\end{eqnarray}
 Therefore,
\begin{eqnarray}
\partial^2_{y_n}R(x,y)=- q(y)(R(x,y)+\Gamma(x,y)) - \Delta'R(x,y)
\end{eqnarray}
and consequently
\begin{eqnarray}\label{nn}
|\partial^2_{y_n}R(x,y)|\le C |x-y|^{2-n}\ .
\end{eqnarray}
Combining \eqref{ij} and \eqref{nn} together we get the desired bound for the full Hessian of $R$ and the thesis follows.
\end{proof}

%%%%%%%%%%%%%%%%%%%%%%%%%%%%%%%%%%%%%%%%%%%%%%%%%%%%%%%%%%%%%%%%%%%%%

\subsection{Propagation of smallness}\label{PS}
\begin{lemma}\label{uc}
Let $v$ be a weak solution to
\begin{equation}
\Delta v + q v=0 \ \ \mbox{in}\ \ \mathcal{W}_k ,
\end{equation}
where $q$ is either equal to $\widetilde{q}^{(1)}$ or equal to $\widetilde{q}^{(2)}$. Assume that for given positive numbers $\epsilon_0, E_0$ and real number $\gamma$, $v$ satisfies
\begin{equation}
|v(x)|\le \epsilon_0\ \ \ \mbox{for every}\ \ x\in\widehat{D}_0
\end{equation}
and
\begin{equation}
|v(x)|\le C(\epsilon_0 + E_0)(1+ \mbox{dist}(x,\mathcal{U}_k))^{\gamma},\ \ \ \mbox{for every}\ \ x\in \mathcal{W}_k ,
\end{equation}

for some positive constant $C$. Then the following inequality holds true for every $0<r<r_1$
\begin{equation}
|v(y_{k+1})|\le C\left(\frac{\epsilon_0}{\epsilon_0 + E_0}\right)^{r\beta^{N_1}}(\epsilon_0 + E_0)(1+r^{\gamma}),
\end{equation}
where $y_{k+1}=P_{k+1}-2r\nu(P_{k+1})$, where $\nu$ is the exterior unit normal to $\partial D_k$ at $P_{k+1}$, $\beta=\frac{\ln(8/7)}{\ln 4}$, $r_1=\frac{r_0}{16}$ and the constants $C, N_1$ depend on the a priori data only.
\end{lemma}

\begin{proof}
By proposition 3.9 in \cite{Beretta2012}, which is based on an iterated use of the three spheres inequality for elliptic equations, we infer that
\begin{equation}
|v(y_{k+1})|\le C\left(\frac{\epsilon_0}{\epsilon_0 + E_0}\right)^{r\beta^{N_1}}(\epsilon_0 + E_0)(1+r^{(1-\tau_r)\gamma}),
\end{equation}
where $\tau_r=\frac{\ln\left(\frac{12r_1 - 2r}{12 r_1 -3r}\right)}{\ln\left(\frac{6r_1 - r}{2r_1}\right)}\in(0,1)$.
By noticing that there exist positive constants $C_1$ and $C_2$ such that
\begin{eqnarray}
C_1 r^{\gamma}\le r^{(1-\tau_r)\gamma}\le C_2 r^{\gamma},
\end{eqnarray}
the thesis follows.
\end{proof}

%%%%%%%%%%%%%%%%%%%%%%%%%%%%%%%%%%%%%%%%%%%%%%%%%%%%%%%%%%

\begin{proof}{of Theorem \ref{proposizione unique continuation finale}}

We observe that for any  $y,z\in\mathcal{W}_k$ we have
\begin{eqnarray}\label{stimadist}
\  \ |S_{\mathcal{U}_k}(y,z)|\le C \|\widetilde{q}^{(1)}-\widetilde{q}^{(2)}\|_{L^{\infty}(\Omega_0)}(\mbox{dist}(y,\mathcal{U}_k)\mbox{dist}(z,\mathcal{U}_k))^{2-\frac{n}{2}},\ \
\end{eqnarray}
for any\ \ $y,z\in \mathcal{W}_k$. By \eqref{stimadist} and applying twice lemma \ref{uc} first to $v(\cdot)=S_{\mathcal{U}_k}(\cdot,z)$, with $z\in (D_0)_{\frac{r_0}{4}}$ and then to $v(\cdot)=S_{\mathcal{U}_k}(y,\cdot)$, with $y\in  B_{3r_1 - r}(x_{k+1})$, we find that
\begin{equation}\label{primastima}
\left|\widetilde{S}_{\mathcal{U}_k}\left(y,z\right)\right|
\leq C\left(\frac{\varepsilon_0}{E+\varepsilon_0}\right)^{r^2\beta^{2N_1}}(1+r^{2\gamma}),
\end{equation}

for any $y,z \in B_{3r_1 - r}(x_{k+1})$, where $x_{k+1}=P_{k+1}-3r_1\nu(P_{k+1})$, $\nu$ is the exterior unit normal to $\partial{D}_k$ and  $\gamma=2-\frac{n}{2}$\ .

%By repeating the argument in \cite{Be-dH-Q},[proof of Proposition 3.9] concerning a careful analysis of unique continuation argument across $K$ discontinuity interfaces in the two variables $y,z$ and based on an iterated use of the three spheres inequality for elliptic equation, we have that for any $y,z \in B_{3r_1 - r}(x_{k+1})$

%\begin{equation}\label{primastima}
%\left|\tilde{S}_{\mathcal{U}_k}\left(y,z\right)\right|
%\leq C\left(\frac{\varepsilon_0}{E+\varepsilon_0}\right)^{\mu^2}(\varepsilon_0+E)r^{-1}
%\end{equation}

%where $x_{k+1}=P_{k+1}-3r_1\nu(P_{k+1})$, $\nu$ is the exterior unit normal to $\partial{D}_k$ where $\mu=\tau_{r}\beta^{N_1}$, $\beta=\frac{\ln(8/7)}{\ln 4}$, $\tau_r=\frac{\ln\left(\frac{12r_1 - 2r}{12 r_1 -3r}\right)}{\frac{6r_1 - r}{2r_1}}\in(0,1)$,  and the constants $N_1, C>0$ depend on the a-priori data only.

By considering ${\widetilde{S}}_{\mathcal{U}_k}({y},z)$ as a function of $2n$ variables where $(y,z)\in \mathbb{R}^{2n}$, \eqref{primastima} leads to

\begin{equation}\label{secstima}
|{\widetilde{S}}_{\mathcal{U}_k}(y_1,\dots, y_n,z_1,\dots, z_{n})|\leq  C\left(\frac{\varepsilon_0}{E+\varepsilon_0}\right)^{r^2\beta^{2N_1}}(1+r^{2\gamma}),
\end{equation}

for any $x=(y_1,\dots, y_n,z_1,\dots, z_{n}) \in B_{3r_1 - r}(x_{k+1})\times \in B_{3r_1 - r}(x_{k+1})$. Now observing that ${\tilde{S}}_{\mathcal{U}_k}(y_1,\dots, y_n,z_1,\dots, z_{n})$ is a solution in $D_k\times D_k$ of the elliptic equation

\begin{equation}
\left(\Delta_y + \Delta_z\right) \widetilde{S}_{\mathcal{U}_k}(y,z) + \left(q_1(y) + q_2(z)\right) \tilde{S}_{\mathcal{U}_k}(y,z) = 0,
\end{equation}

we have that given $y_{k+1}=P_{k+1}-2r\nu(P_{k+1})$ by Schauder interior estimates

\begin{eqnarray*}
& &\|\partial_{y_i}\partial_{z_j}{\tilde{S}}_{\mathcal{U}_k}(y_1,\dots, y_n,z_1,\dots, z_{n})\|_{L^{\infty}(B_{\frac{r}{2}}(y_{k+1})\times \in B_{\frac{r}{2}}(y_{k+1})}\nonumber\\
& &\le\frac{C}{r^2}\|{\tilde{S}}_{\mathcal{U}_k }(y_1,\dots, y_n,z_1,\dots, z_{n})\|_{L^{\infty}( B_{ r}(y_{k+1})\times \in B_{ r}(y_{k+1})}
\end{eqnarray*}

%Moreover, we have that being $d_{\bar{h}(r)-1}>r$,  hence it follows $r<\frac{d_0}{a\rho_0}\rho_{\bar{h}(r)}$, which in turn leads to
%\begin{eqnarray}\label{Sh}
%&&\|\partial_{x_i}\partial_{x_j}{\tilde{S}}_{\mathcal{U_K}}(x_1,\dots, x_{2n})\|_{L^{\infty}(\tilde{Q}_{\frac{{\rho_{\bar{h}(r)}}}{2}}(w_{\bar{h}(r)}(Q_{k+1})))} \nonumber\\
%&&\le\frac{C}{r^2}\|{\tilde{S}}_{\mathcal{U_K}}(x_1,\dots, x_{2n})\|_{L^{\infty}(\tilde{Q}_{{{\rho_{\bar{h}(r)}}}}(w_{\bar{h}(r)}(Q_{k+1})))}
%\end{eqnarray}

and

\begin{eqnarray*}
& &\|\partial^{2}_{y_i}\partial^{2}_{z_j}{\tilde{S}}_{\mathcal{U}_k}(y_1,\dots, y_n,z_1,\dots, z_{n})\|_{L^{\infty}(B_{\frac{r}{4}}(y_{k+1})\times \in B_{\frac{r}{4}}(y_{k+1})}\nonumber\\
& &\le\frac{C}{r^2}\|\partial_{y_i}\partial_{z_j}{\tilde{S}}_{\mathcal{U}_k}(y_1,\dots, y_n,z_1,\dots, z_{n})\|_{L^{\infty}( B_{\frac{r}{2}}(y_{k+1})\times \in B_{\frac{r}{2}}(y_{k+1})},
\end{eqnarray*}

for any $i,j=1,\dots, n$, where $C>0$ is a constant depending on the a priori data only.

%Noticing that
%\begin{eqnarray}\label{Sh2}
%\frac{\log({r/r_0})}{\log(a)}\le \bar{h}(r)\le \frac{\log({r/r_0})}{\log(a)} +1
%\end{eqnarray}
%we find that
%\begin{eqnarray}
%r^{-2}\le \left(\frac{a}{r_0}\right)^2\left(\frac{1}{a^2}\right)^{\bar{h}(r)}\ .
%\end{eqnarray}

%Finally by combining \eqref{secstima}, \eqref{Sh} and the above inequality we get the desired estimate.

\end{proof}

%%%%%%%%%%%%%%%%%         STABILITY OF q AND NORMAL DERIVATIVE AT THE BOUNDARY         %%%%%%%%%%%%%%%%%%

\subsection{Stability at the boundary}\label{stability at the boundary}

\begin{proof}[Proof of estimate \eqref{estimate appendix}]

%\begin{proposition}

%\end{proposition}

We choose a coordinate system $\{x_1,\dots , x_n\}$ centred at $P_1$ with $x_n$ in the direction of the normal $\nu$ and recall that for every $y,z\in D_0$ we have

\begin{eqnarray}\label{Alessandrini 2 appendix}
& &\int_{\partial\Omega}\!\!\left(\partial_{y_n}\widetilde{G}_1(x,y)\partial_{x_n}\partial_{z_n}\widetilde{G}_2(x,z)-\partial_{z_n}\widetilde{G}_2(x,z)\partial_{x_n}\partial_{y_n}\widetilde{G}_1(x,y)\right)\:dS(x)\\
& & = \int_{\Omega}(\widetilde{q}^{(1)}-\widetilde{q}^{(2)})(x)\partial_{y_n}\widetilde{G}_1(x,y)\partial_{z_n}\widetilde{G}_2(x,z)\:dx = \partial_{y_n}\partial_{z_n}\widetilde{S}_{\mathcal{U}_{0}}(y,z).\nonumber
\end{eqnarray}

and

\begin{eqnarray}\label{Alessandrini 3 appendix}
& &\int_{\partial\Omega}\!\!\left(\partial^{2}_{y_n}\widetilde{G}_1(x,y)\partial_{x_n}\partial^{2}_{z_n}\widetilde{G}_2(x,z)-\partial^{2}_{z_n}\widetilde{G}_2(x,z)\partial_{x_n}\partial^{2}_{y_n}\widetilde{G}_1(x,y)\right)\:dS(x)\\
& & =\int_{\Omega}(\widetilde{q}^{(1)}-\widetilde{q}^{(2)})(x)\partial^{2}_{y_n}\widetilde{G}_1(x,y)\partial^{2}_{z_n}\widetilde{G}_2(x,z)\:dx
= \partial^{2}_{y_n}\partial^{2}_{z_n}\widetilde{S}_{\mathcal{U}_{0}}(y,z).\nonumber
\end{eqnarray}

%\noindent\textbf{qua va controllata la derivazione sotto segno di integrale a sinistra}

%and that, for every $k\in\{1,\dots K\}$,

%\[\partial_{y_n}\partial_{z_n}\tilde{S}_{\mathcal{U}_{k-1}}(y,z)=\int_{\mathcal{U}_{k-1}}
%(\tilde\gamma^{(1)}-\tilde\gamma^{(2)})
%\partial_{y_n}\nabla\tilde{G}_1(\cdot,y)\cdot\partial_{z_n}\nabla\tilde{G}_2(\cdot,z),\]

By combining \eqref{Alessandrini 2 appendix} together with \eqref{norm Cauchy data}, \eqref{Alessandrini stability}, we obtain

\begin{eqnarray}
& &\left|\int_{\partial\Omega}\!\!\left(\partial_{y_n}\widetilde{G}_1(x,y)\partial_{x_n}\partial_{z_n}\widetilde{G}_2(x,z)-\partial_{z_n}\widetilde{G}_2(x,z)\partial_{x_n}\partial_{y_n}\widetilde{G}_1(x,y)\right)\:dS(x)\right|\\
& &\leq  C \varepsilon_0 \left(d(y)d(z)\right)^{-\frac{n}{2}},\qquad
\textnormal{for\:every}\:y,z\in D_0,\nonumber
\end{eqnarray}

where $d(y)$ denotes the distance of $y$ from $\Omega$ and $C$ is a constant that depends on $L$, $\lambda$ and $n$ only. Let $\rho_0=\frac{r_0}{{C}}$, where ${C}$ is the constant introduced in Theorem \ref{ordine01},  let $r\in(0,r_1)$, where $r_1$ has been introduced in Proposition \ref{proposizione unique continuation finale} and denote

\[y_1=P_1+r\nu.\]

We set $y=z=y_1$ and obtain

\begin{eqnarray}%\label{S=I1+I2}
& &\int_{\Omega}(\widetilde{q}^{(1)}-\widetilde{q}^{(2)})(x)\partial_{y_n}\widetilde{G}_1(x,y)\partial_{z_n}\widetilde{G}_2(x,z)\:dx\\
& &=
\int_{B_{\rho_0}(P_1)\cap D_1}(\widetilde{q}^{(1)}-\widetilde{q}^{(2)})(x)\partial_{y_n}\widetilde{G}_1(x,y)\partial_{z_n}\widetilde{G}_2(x,z)\:dx\nonumber\\
& &+ \int_{\Omega\setminus (B_{\rho_0}(P_1)\cap
D_1)}(\widetilde{q}^{(1)}-\widetilde{q}^{(2)})(x)\partial_{y_n}\widetilde{G}_1(x,y)\partial_{z_n}\widetilde{G}_2(x,z)\:dx,\nonumber
\end{eqnarray}

which leads to

\begin{eqnarray}\label{S=I1+I2 inq 1}
\varepsilon_0 r^{-n}& \geq  &\left| \int_{B_{\rho_0}(P_1)\cap D_1}(q^{(1)}_1- q^{(2)}_1)(\cdot)\partial_{y_n}\widetilde{G}_1(x,y_1)\partial_{z_n}\widetilde{G}_2(x,y_1)\right|\nonumber\\
& - &\left|\int_{\Omega\setminus (B_{\rho_0}(P_1)\cap
D_1)}(q^{(1)}-q^{(2)})(\cdot)
\partial_{y_n}\widetilde{G}_1(x,y_1)\partial_{z_n}\widetilde{G}_2(x,y_1)\right|.
\end{eqnarray}

Let $x^{0}\in\overline{\Sigma_1\cap B_{\frac{r_0}{4}(P_1)}}$ such that $\left(q^{(1)}_1 - q^{(2)}_1\right)(x^0) = ||\widetilde{q}^{(1)}-\widetilde{q}^{(2)}||_{L^{\infty}(\Sigma_1\cap B_{\frac{r_0}{4}(P_1)})}$ and recall that $\left(q^{(1)}_1 - q^{(2)}_1\right)(x)=\alpha_1+\beta_1\cdot x$, therefore we obtain

%\begin{equation}\label{stima I2}
%|I_2(w)|\leq CE\rho_0^{-n},
%\end{equation}

\begin{eqnarray}\label{S=I1+I2 inq 2}
\varepsilon_0 r^{-n}  & \geq & \left| \int_{B_{\rho_0}(P_1)\cap D_1}(q^{(1)}_1-q^{(2)}_1)(x^0)\partial_{y_n}\widetilde{G}_1(x,y_1)\partial_{z_n}\widetilde{G}_2(x,y_1)\right|\nonumber\\
&-& \left|\int_{B_{\rho_0}(P_1)\cap D_1}\beta_1\cdot (x-x^0)\partial_{y_n}\widetilde{G}_1(x,y_1)\partial_{z_n}\widetilde{G}_2(x,y_1)\right|-CE\rho_0^{2-n}
\end{eqnarray}

and then

\begin{eqnarray}\label{S=I1+I2 inq 3}
& & ||\widetilde{q}_1^{(1)}-\widetilde{q}_1^{(2)}||_{L^{\infty}(\Sigma_1\cap B_{\frac{r_0}{4}}(P_1))}\left|\int_{B_{\rho_0}(P_1)\cap D_1} \partial_{y_n}\widetilde{G}_1(x,y_1)\partial_{z_n}\widetilde{G}_2(x,y_1)\:dx\right|\nonumber\\
& &\leq\int_{B_{\rho_0}(P_1)\cap D_1}|\beta_1|\:|x-x^0|\:|\partial_{y_n}\widetilde{G}_1(x,y_1)| |\partial_{z_n}\widetilde{G}_2(x,y_1)|\:dx\nonumber\\
& & + CE\rho_0^{2-n}+\varepsilon_0 r^{-n}.
\end{eqnarray}

For $n=3$, by combining \eqref{S=I1+I2 inq 3} together with \eqref{3b}, we obtain

\begin{eqnarray}\label{S=I1+I2 inq 4 3D}
& & ||\widetilde{q}_1^{(1)}-\widetilde{q}_1^{(2)}||_{L^{\infty}(\Sigma_1\cap B_{\frac{r_0}{4}(P_1)})}\int_{B_{\rho_0}(P_1)\cap D_1} |\nabla\Gamma(x-y_1)|^2\:dx\nonumber\\
 & &\leq C\bigg\{ E\int_{B_{\rho_0}(P_1)\cap D_1}|\nabla\Gamma(x-y_1)|\:\log|x-y_1|\:dx\nonumber\\
 & & +  E\int_{B_{\rho_0}(P_1)\cap D_1}\left(\log|x-y_1|\right)^{2}\:dx\nonumber\\
& &  + E \int_{B_{\rho_0}(P_1)\cap D_1}\:|x-x^{0}||\:x-y_1|^{-4}\:dx + E\rho_0^{-1}+\varepsilon_0 r^{-3}\bigg\},
\end{eqnarray}

which leads to

\begin{eqnarray}\label{S=I1+I2 inq 4 bis 3D}
& & ||\widetilde{q}_1^{(1)}-\widetilde{q}_1^{(2)}||_{L^{\infty}(\Sigma_1\cap B_{\frac{r_0}{4}(P_1)})}\int_{B_{\rho_0}(P_1)\cap D_1} |\:x-y_1|^{-4}\:dx\nonumber\\
 & &\leq C\bigg\{ E\int_{B_{\rho_0}(P_1)\cap D_1}|x-y_1|^{-3}\:dx +
E\int_{B_{\rho_0}(P_1)\cap D_1}|\:x-y_1|^{-2}\:dx \nonumber\\
& &  + E\int_{B_{\rho_0}(P_1)\cap D_1}|x-y_1|^{-3}\:dx  + E\rho_0^{-1}+\varepsilon_0 r^{-3}\bigg\},
\end{eqnarray}

therefore

\begin{eqnarray}
||\widetilde{q}^{(1)}-\widetilde{q}^{(2)}||_{L^{\infty}(\Sigma_1\cap B_{\frac{r_0}{4}(P_1)})} &\leq & C\left\{E r\log\left(\frac{\rho_0}{r}\right) + Er^{2} + E\rho^{-1}_0  r+ \varepsilon_0 r^{-2}.\right\}\nonumber\\
&\leq & C\left\{Er^{\theta}+\varepsilon r^{-2}\right\},
\end{eqnarray}

for some $\theta$, $0<\theta<1$ and by optimizing with respect to $r$

\begin{equation}\label{stima lipschitz gamma 1}
||\widetilde{q}^{(1)}-\widetilde{q}^{(2)}||_{L^{\infty}(\Sigma_1\cap B_{\frac{r_0}{4}(P_1)})}\leq C \varepsilon^{\frac{\theta}{\theta +2}}_0 \left(E+\varepsilon_0\right)^{\frac{2}{\theta+2}}.
\end{equation}

For $n\geq 4$, by combining \eqref{S=I1+I2 inq 3} together \eqref{3b}, we obtain

\begin{eqnarray}\label{S=I1+I2 inq 4 big D}
& & ||\widetilde{q}_1^{(1)}-\widetilde{q}_1^{(2)}||_{L^{\infty}(\Sigma_1\cap B_{\frac{r_0}{4}(P_1)})}\int_{B_{\rho_0}(P_1)\cap D_1} |\nabla\Gamma(x-y_1)|^2\:dx\nonumber\\
 & &\leq C\bigg\{ E\int_{B_{\rho_0}(P_1)\cap D_1}|\nabla\Gamma(x-y_1)|\:|x-y_1|^{3-n}\:dx\nonumber\\
 & & +  E\int_{B_{\rho_0}(P_1)\cap D_1}\:|x-y_1|^{6-2n}\:dx\nonumber\\
& &  + E \int_{B_{\rho_0}(P_1)\cap D_1}\:|x-x^{0}||\:x-y_1|^{2-2n}\:dx + E\rho_0^{2-n}+\varepsilon_0 r^{-n}\bigg\},
\end{eqnarray}

which leads to

\begin{eqnarray}\label{S=I1+I2 inq 4 bis big D}
& & ||\widetilde{q}_1^{(1)}-\widetilde{q}_1^{(2)}||_{L^{\infty}(\Sigma_1\cap B_{\frac{r_0}{4}(P_1)})}\int_{B_{\rho_0}(P_1)\cap D_1} |\:x-y_1|^{2-2n}\:dx\nonumber\\
 & &\leq C\bigg\{ E\int_{B_{\rho_0}(P_1)\cap D_1}|x-y_1|^{4-2n}\:dx +
E\int_{B_{\rho_0}(P_1)\cap D_1}|\:x-y_1|^{6-2n}\:dx \nonumber\\
& &  + E\int_{B_{\rho_0}(P_1)\cap D_1}|x-y_1|^{3-2n}\:dx  + E\rho_0^{2-n}+\varepsilon_0 r^{-n}\bigg\},
\end{eqnarray}

therefore

\begin{eqnarray}
||\widetilde{q}^{(1)}-\widetilde{q}^{(2)}||_{L^{\infty}(\Sigma_1\cap B_{\frac{r_0}{4}(P_1)})} &\leq & C\left\{E r^{2} + Er^{4} + Er + E\rho^{2-n}_0  r+ \varepsilon_0 r^{-2}.\right\}\nonumber\\
&\leq & C\left\{Er+\varepsilon_0 r^{-2}\right\}
\end{eqnarray}

and by optimizing with respect to $r$

\begin{equation}\label{stima lipschitz gamma}
||\widetilde{q}^{(1)}-\widetilde{q}^{(2)}||_{L^{\infty}(\Sigma_1\cap B_{\frac{r_0}{4}(P_1)})}\leq C \varepsilon^{\frac{1}{3}}_0 \left(E+\varepsilon_0\right)^{\frac{2}{3}}.
\end{equation}

We proceed by estimating $\partial_{\nu}\left(\widetilde{q}^{(1)}-\widetilde{q}^{(2)}\right)(P_{1})$.  By combining \eqref{Alessandrini 3 appendix} together with \eqref{norm Cauchy data}, \eqref{Alessandrini stability}, we obtain

\begin{eqnarray}\label{Alessandrini 3 appendix a}
& &\left|\int_{\partial\Omega}\!\!\left(\partial^{2}_{y_n}\widetilde{G}_1(x,y)\partial_{x_n}\partial^{2}_{z_n}\widetilde{G}_2(x,z)-\partial^{2}_{z_n}\widetilde{G}_2(x,z)\partial_{x_n}\partial^{2}_{y_n}\widetilde{G}_1(x,y)\right)\:dS(x)\right|\\
& & \leq  C \varepsilon_0 \left(d(y)d(z)\right)^{-\frac{n}{2}-1},\qquad
\textnormal{for\:every}\:y,z\in D_0,\nonumber
\end{eqnarray}

and setting $y=z=y_1$ in \eqref{Alessandrini 3 appendix a}, we obtain

\begin{eqnarray}%\label{S=I1+I2}
& &\int_{\partial\Omega}\!\!\left(\partial^{2}_{y_n}\widetilde{G}_1(x,y)\partial_{x_n}\partial^{2}_{z_n}\widetilde{G}_2(x,y_1)-\partial^{2}_{z_n}\widetilde{G}_2(x,y_1)\partial_{x_n}\partial^{2}_{y_n}\widetilde{G}_1(x,y_1)\right)\:dS(x)\\
& &=I_1(y_1)+I_2(y_1),\nonumber
\end{eqnarray}

where

\[I_1(y_1)=\int_{B_{\rho_0}(P_1)\cap D_1}(\widetilde{q}^{(1)}-\widetilde{q}^{(2)})(x)\partial^{2}_{y_n}\widetilde{G}_1(x,y_1)\partial^{2}_{z_n}\widetilde{G}_2(x,y_1)\:dx,\]

\[I_2(y_1)=\int_{\Omega\setminus (B_{\rho_0}(P_1)\cap D_1)}(\widetilde{q}^{(1)}-\widetilde{q}^{(2)})(x)\partial^{2}_{y_n}\widetilde{G}_1(x,y_1)\partial^{2}_{z_n}\widetilde{G}_2(x,y_1)\:dx,\]

and

\begin{equation}%\label{stima I2}
|I_2(w)|\leq CE\rho_0^{-n-2}.
\end{equation}

We have

\begin{eqnarray*}
& &|I_1(y_1)|\\
& &\geq \left|\int_{B_{\rho_0}(P_1)\cap D_1}\!\!(\partial_{x_n}(q^{(1)}_1-q^{(2)}_1)(P_1))(x-P_1)_n\partial^2_{y_n}\widetilde{G}_1(x,y_1)
\partial^2_{z_n}\widetilde{G}_2(x,y_1)\right|\:dx\nonumber\\
& &-\int_{B_{\rho_0}(P_1)\cap D_1}\!\!\!\!|(D_T(q^{(1)}_1-q^{(2)}_1)(P_1))\cdot (x-P_1)'||\partial^2_{y_n}\widetilde{G}_1(x,y_1)|\:|
\partial^2_{z_n}\widetilde{G}_2(x,y_1)|\:dx\nonumber\\
& &-\int_{B_{\rho_0}(P_1)\cap D_1}|(q^{(1)}_1-q^{(2)}_1)(P_1)||\partial^2_{y_n}\widetilde{G}_1(x,y_1)|\:|
\partial^2_{z_n}\widetilde{G}_2(x,y_1)|\:dx.
\end{eqnarray*}

Noticing that up to a transformation of coordinates we can assume that $P_1$ coincides with the origin $O$ of the coordinates system and recalling Theorem \ref{ordine2}, this leads to

\begin{eqnarray}%\label{stima S}
|I_1(y_1)|
& &\geq
|\partial_{x_n}(q^{(1)}_1-q^{(2)}_1)(O)|C\int_{B_{\rho_0}(O)\cap
D_1}|\partial^2_{y_n}\Gamma(x,y_1)|^2\:|x_n|\:dx\nonumber\\
& &- C\bigg\{E\int_{B_{\rho_0}(O)\cap
D_1}|\partial^2_{y_n}\Gamma(x,y_1)|\:|x-y_1|^{2-n}|x_n|\:dx\nonumber\\
& &-E\int_{B_{\rho_0}(O)\cap D_1}|x-y_1|^{4-2n}|x_n|\:dx\bigg\}\nonumber\\
& &-\int_{B_{\rho_0}(O)\cap D_1}\!\!\!\!|(D_T(q^{(1)}_1-q^{(2)}_1)(O))|\:|x'|\:|\partial^2_{y_n}\widetilde{G}_1(x,y_1)|\:|
\partial^2_{z_n}\widetilde{G}_2(x,y_1)|\:dx\nonumber\\
& &-\int_{B_{\rho_0}(O)\cap D_1}|(q^{(1)}_1-q^{(2)}_1)(O)|\:|\partial^2_{y_n}\widetilde{G}_1(x,y_1)|\:|
\partial^2_{z_n}\widetilde{G}_2(x,y_1)|\:dx.\end{eqnarray}

Therefore, by combining \eqref{stima S} together with \eqref{S=I1+I2} and \eqref{stima
I2}, we obtain

\begin{eqnarray*}
|I_1(y_1)| &\geq &
|\partial_{x_n}(q^{(1)}_1-q^{(2)}_1)(O)|C\int_{B_{\rho_0}(P_1)\cap
D_1}|x-y_1|^{1-2n}\:dx\noindent\\
&-&C\bigg\{E\int_{B_{\rho_0}(O)\cap
D_1}|x-y_1|^{3-2n}\:dx\noindent\\
&-&E\int_{B_{\rho_0}(O)\cap
D_1}|x-y_1|^{5-2n}\:dx\nonumber\\
&-&(\varepsilon_0 + E)\left(\frac{\varepsilon_0}{\varepsilon_0 + E}\right)^{\eta_1}\int_{B_{\rho_0}(O)\cap D_1}\:|x-y_1|^{1-2n}\:dx\nonumber\\
&-&(\varepsilon_0 + E)\left(\frac{\varepsilon_0}{\varepsilon_0 + E}\right)^{\eta_1}\int_{B_{\rho_0}(O)\cap D_1}\:|x-y_1|^{-2n}\:dx\bigg\},
\end{eqnarray*}

which implies

\begin{equation}
|\partial_{x_n}(\gamma^{(1)}_1-\gamma^{(2)}_1)(O)|\sigma^{1-n}\le |I_1(y_1)| + C\Big\{E r^{3-n} + (\varepsilon_0 + E)\left(\frac{\varepsilon_0}{\varepsilon_0 + E}\right)^{\eta_1} r^{-n}\Big\},
\end{equation}

and

\begin{eqnarray}\label{l1}
|I_1(y_1)| \le& &\left|\int_{\partial\Omega}\left(\partial^{2}_{y_n}\widetilde{G}_1(x,y_1)\partial_{x_n}\partial^{2}_{z_n}\widetilde{G}_2(x,y_1)-\partial^{2}_{z_n}\widetilde{G}_2(x,y_1)\partial_{x_n}\partial^{2}_{y_n}\widetilde{G}_1(x,y_1)\right)\:dS(x)\right|\nonumber\\
 &&\ \ \ \ \ \ + C E \rho_0^{-n-2}.
\end{eqnarray}
%\sigma^{-n}C^{\bar{h}}(E+\varepsilon +\delta_{k-1})\omega_{\frac{1}{c}}(...) + C E \rho_0^{-n}

Thus by combining together the last two inequalities we get

\begin{eqnarray}
|\partial_{x_n}(q^{(1)}_1-q^{(2)}_1)(O)| r^{1-n} &\le &  C \Big(\varepsilon_0 r^{-n-2}+ E \rho_0^{-n-2}\nonumber\\
&+&E r^{3-n} + (\varepsilon_0 + E)\left(\frac{\varepsilon_0}{\varepsilon_0 + E}\right)^{\eta_1} r^{-n}\Big),
\end{eqnarray}

therefore

\begin{equation}
|\partial_{x_n}(q^{(1)}_1-q^{(2)}_1)(O)| \le   C \left\{ \varepsilon_0 r^{-3} + E r^{2} + (E+\varepsilon_0)\left(\frac{\varepsilon_0}{E+\varepsilon_0}\right)^{\eta_1}r^{-1}\right\}
\end{equation}

and by optimizing with respect to $r$ we get

\begin{equation}
|\partial_{x_n}(q^{(1)}_1-q^{(2)}_1) (O)| \le   C (E+\varepsilon_0)\left(\frac{\varepsilon_0}{E+\varepsilon_0}\right)^{\frac{2\eta_1}{5}}.
\end{equation}

\end{proof}

\section*{\normalsize{Acknowledgments}}
The research carried out by G. Alessandrini and E. Sincich for the preparation of this paper has been supported by FRA 2016 "Problemi inversi, dalla stabilit\`{a} alla ricostruzione" funded by Universit\`{a} degli Studi di Trieste. E. Sincich has been also supported by  Gruppo Nazionale per l'Analisi Matematica, la Probabilit\`{a} e le loro Applicazioni  (GNAMPA) by the grant " Problemi Inversi per Equazioni Differenziali''.   E. Sincich is grateful for the support and the hospitality of the  Department of Mathematics and Statistics of the University of Limerick, where part of this work has been carried over.  R. Gaburro and E. Sincich acknowledge the support of {{"Programma professori visitatori''}}, Istituto Nazionale di Alta Matematica Francesco Severi (INdAM) during the Fall 2016/17. R. Gaburro wishes to acknolwedge also the support of MACSI, the Mathematics Applications Consortium for Science and Industry (www.macsi.ul.ie), funded by the Science Foundation Ireland Investigator Award 12/IA/1683. M.V de Hoop was partially supported by the Simons Foundation under the MATH $+$ X program, the National Science Foundation under grant DMS-1559587, and by the members of the Geo-Mathematical Group at Rice University.

%%%%%%%%%%%%%%%%%%%%%%%%%%%%%%%%%%%%%%%%%%%%%%%%%%%%%%%%%%%%%%%%%%%%%%%%%%%%%%%%%%%%%%%%%%%

\end{document}